\documentclass{amsart}  

\usepackage{amsfonts}
\usepackage{amsthm}
\usepackage{amsmath}
\usepackage{amsfonts}
\usepackage{latexsym}
\usepackage{amssymb}
\usepackage{amscd}
\usepackage[latin1]{inputenc}
\usepackage{verbatim}

\newtheorem{theorem}{Theorem}[section]
\newtheorem{lemma}[theorem]{Lemma}
\newtheorem{proposition}[theorem]{Proposition}
\newtheorem{corollary}[theorem]{Corollary}
\newtheorem{example}[theorem]{Example}

\theoremstyle{definition}
\newtheorem{definition}{Definition}[section]

\newtheorem{notation}{Notation}[section]

\theoremstyle{remark}
\newtheorem{remark}{Remark}

\newcommand\hil{\mathcal H}                        
\newcommand\hilk{\mathcal K}                       
\newcommand\ssL{\mathcal L}                       
\newcommand\bh{\mathcal B(\mathcal H)}    
\newcommand\bk{\mathcal B(\mathcal K)}     
\newcommand\bhpi{\mathcal B(\mathcal H_\pi)}

\newcommand\mn{{\mathcal M}_n}
\newcommand\mm{{\mathcal M}_m}
\newcommand\mk{{\mathcal M}_k}

\newcommand\mpee{{\mathcal M}_p}
\newcommand\mr{{\mathcal M}_r}

\newcommand\mtwo{{\mathcal M}_2}

\newcommand\tn{{\mathcal T}_n}            

\newcommand\fn{{\mathbb F}_n}         
\newcommand\fnl{{\mathbb F}_{n-1}}
\newcommand\fni{{\mathbb F}_{\infty}}

\newcommand\ose{{\mathcal E}}

\newcommand\osi{{\mathcal I}}
\newcommand\osr{{\mathcal R}}
\newcommand\oss{{\mathcal S}}
\newcommand\ost{{\mathcal T}}
\newcommand\osw{{\mathcal W}}
\newcommand\osu{{\mathcal U}}
\newcommand\osv{{\mathcal V}}
\newcommand\kj{{\mathcal J}}                    
\newcommand\ie{{\mathcal I(\mathcal S)}}  

\newcommand\cminp{{ C}^{\rm min}_p}

\newcommand\cmaxp{{C}^{\rm max}_p}

\newcommand\dmaxp{{D}^{\rm max}_p}
\newcommand\dmaxone{{D}^{\rm max}_1}
\newcommand\ccomp{{ C}^{\rm comm}_p}

\newcommand\omin{\otimes_{\rm min}}
\newcommand\omax{\otimes_{\rm max}}
\newcommand\oc{\otimes_{\rm c}}

\newcommand\pstar{{\rm (}\mathfrak W{\rm )}}   
\newcommand\pstarn{{\rm (}\mathfrak W_n{\rm )}} 
\newcommand\ps{{\rm (}\mathfrak S{\rm )}}   
\newcommand\psn{{\rm (}\mathfrak S_n{\rm )}}

\newcommand\csta{{\mathcal A}}
\newcommand\cstb{{\mathcal B}}
\newcommand\ia{{\mathcal I(A)}}
\newcommand\cstar{{\rm C}^*}                              
\newcommand\cstare{{\rm C}_{\rm e}^*}               

\title[Operator System Quotients and Tensor Products]{Operator System Quotients of Matrix Algebras and their Tensor Products}

\author{Douglas Farenick}
\address{Department of Mathematics and Statistics, University of Regina,
Regina, Saskatchewan S4S 0A2, Canada}

\author{Vern I.~Paulsen}
\address{Department of Mathematics, University of Houston,
Houston, Texas 77204-3476, U.S.A.}

\begin{document}
\maketitle

\begin{abstract} If  $\phi:\oss\rightarrow\ost$ 
is a completely positive (cp) linear map
of operator systems and if $\kj=\ker\phi$, then the quotient
vector space $\oss/\kj$ may be endowed with a matricial ordering through
which $\oss/\kj$ has the structure of an operator system. Furthermore,
there is a uniquely determined cp map $\dot{\phi}:\oss/\kj \rightarrow\ost$ such that
$\phi=\dot{\phi}\circ q$, where $q$ is the canonical
linear map of $\oss$ onto $\oss/\kj$. The cp map $\phi$ is called a
complete quotient map if $\dot{\phi}$ is a complete order isomorphism between the
operator systems $\oss/\kj$ and $\ost$. Herein we study certain quotient
maps in the cases where $\oss$ is a full matrix algebra or a full subsystem of
tridiagonal matrices. 

Our study of operator system quotients of matrix algebras and tensor products has applications to operator algebra
theory. In particular, we give a new, simple proof of Kirchberg's Theorem
$\cstar(\fni)\omin\bh=\cstar(\fni)\omax\bh$, show that an affirmative solution to the Connes Embedding Problem 
is implied by various matrix-theoretic problems, and give a new characterisation of  
unital C$^*$-algebras that have the weak expectation property.
\end{abstract}
 
\section*{Introduction}

The C$^*$-algebra $\bh$ of bounded linear operators acting on a Hilbert 
space $\hil$ and the group C$^*$-algebra $\cstar(\fni)$ of the free group $\fni$
with countably infinitely many generators are both universal objects in operator
algebra theory. Therefore, 
it is a remarkable fact that  $\cstar(\fni)\omin\bh\,=\,\cstar(\fni)\omax\bh$, which
is a well known and important theorem of Kirchberg \cite[Corollary 1.2]{kirchberg1994}. Kirchberg's proof
was achieved by first showing that the C$^*$-algebra $\cstar(\fni)$ has the lifting property \cite[Lemma 3.3]{kirchberg1994}
and by then invoking his theorem  \cite[Theorem 1.1]{kirchberg1994} that states
$\csta\omin\bh\,=\,\csta\omax\bh$ for every separable unital C$^*$-algebra $\csta$ that has the
lifting property.  A more direct proof of Kirchberg's theorem 
on the uniqueness of the C$^*$-norm for $\cstar(\fni)\otimes\bh$ 
was later found by Pisier \cite{pisier1996},\cite[Chapter 13]{Pisier-book} using 
a mix of operator algebra and operator space theory. Boca \cite{boca1997} made a further extension, replacing $\cstar(\fni)$
with a free product of C$^*$-algebras $\csta_i$ in which each of the maps ${\rm id}_{\csta_i}$ is locally liftable. An exposition of the
results of Kirchberg and Pisier can be found in \cite{ozawa2004}.

One of our main results in this paper is a new proof of Kirchberg's theorem 
(see Corollary \ref{k thm}), obtained herein by reducing Kirchberg's theorem 
to the verification of a certain property (Theorem \ref{pstar thm})
of a finite-dimensional quotient operator system $\osw_n$ whose C$^*$-envelope is $\cstar(\fnl)$.

Our study of operator system quotients of the matrix algebra $\mn$ and of the full operator subsystem $\tn\subseteq\mn$
of tridiagonal matrices (full in the sense that $\cstar(\tn)=\mn$)
allows us to formulate matrix-theoretic questions in Section \S\ref{S:kp-cep}
whose resolution in the affirmative would result in a solution to the Connes Embedding Problem. Our
approach to this is by way of another celebrated theorem of Kirchberg \cite[Proposition 8]{kirchberg1993}:
$\cstar(\fni)\omin\cstar(\fni)=\cstar(\fni)\omax\cstar(\fni)$ if and only if 
every separable II${}_1$-factor can be embedded as a subfactor of the ultrapower 
of the hyperfinite II${}_1$-factor. 
It is of course an open problem whether these equivalent statements are true. The problem of whether or not
$\cstar(\fni)\omin\cstar(\fni)=\cstar(\fni)\omax\cstar(\fni)$ is known as the \emph{Kirchberg Problem}, while the
latter problem involving separable II${}_1$-factors is the \emph{Connes Embedding Problem}.

The methods we use in this paper draw upon recent work of the second author and others 
 \cite{kavruk--paulsen--todorov--tomforde2009,kavruk--paulsen--todorov--tomforde2010,paulsen--todorov--tomforde2010}
on matricially ordered vector spaces, tensor products of operator systems, and quotients
of operator systems.

An operator system is a triple consisting of: (i) all  $*$-vector spaces $\mn(\oss)$ of
$n\times n$ matrices over a fixed $*$-vector space $\oss$;
(ii) distinguished cones $\mn(\oss)_+$ in $\mn(\oss)$ that give rise to a matricial ordering of $\oss$; and (iii)
a distinguished element $e\in\oss$ which is an Archimedean order unit.
The axioms for operator systems are given in \cite{choi--effros1977,Paulsen-book}, and so we will not repeat them here. However,
recall a very important and fundamental fact \cite{choi--effros1977}: 
every operator system $\oss$ arises as an operator subsystem of $\bh$, for some Hilbert space
$\hil$. This result is known as the Choi--Effros Theorem.

Two important C$^*$-algebras that arise in connection with a given operator system $\oss$ are the C$^*$-envelope $\cstare(\oss)$
and the injective envelope $\ie$ of $\oss$. The latter algebra $\ie$ is very large (usually nonseparable) and there exist embeddings so that
$\oss\subseteq\cstar_e(\oss)\subseteq\ie$ and such that $\cstare(\oss)$ is the C$^*$-subalgebra of $\ie$ generated by $\oss$.
These are ``enveloping'' algebras in the sense that $\cstare(\oss)$ is a quotient of every C$^*$-algebra generated by (a copy of) $\oss$
and $\ie$ rigidly contains $\oss$, which is to say that if a linear unital completely positive (ucp) map $\omega:\ie\rightarrow\ie$ satisfies
$\omega_{\vert \oss}={\rm id}_{\oss}$, then necessarily $\omega={\rm id}_{\ie}$. See \cite[Chapter15]{Paulsen-book} for further details on
C$^*$- and injective envelopes.
 
In contrast to progress on tensor products and quotients in operator space theory (see \cite{Pisier-book}, for example),
an analogous theory in the category of operator systems is only now emerging. In \S\ref{S:review} we
review those aspects of the theory that are required for our work herein. In Sections \S\ref{S:osw} and \S\ref{S:oss} we study
certain quotient operator systems that arise from quotients of operator systems of finite matrices, and in Sections \S\ref{S:osw tensor products}
and \S\ref{S:kp-cep} we consider tensor products of these quotients and obtain the earlier mentioned applications to 
Kirchberg's Theorem and the Connes Embedding Problem. In Section \S\ref{S:WEP} we use complete quotient maps to characterise  
unital C$^*$-algebras that possess the weak expectation property (WEP), and in
\S\ref{S:ie} the injective envelope of $\cstar(\fn)$ is determined.

 \section{Tensor Products, Quotients, and Duals of Operator Systems: A Review}\label{S:review}
 
 In this section we review
 some of the fundamental facts, established in
 \cite{kavruk--paulsen--todorov--tomforde2009,kavruk--paulsen--todorov--tomforde2010}, concerning
 tensor products, quotients, and duals of operator systems,
 and introduce the notion of a complete quotient map.
 
Some basic notation: (i)
the Archimedean order unit $e$ of an operator system $\oss$
is generally denoted by $1$, but we will sometimes revert to the use of $e$ in cases where the order unit 
is not canonically given (for example, when considering duals of operator systems); 
(ii) for a linear map $\phi:\oss\rightarrow\ost$, the map $\phi^{(n)}:\mn(\oss)\rightarrow\mn(\ost)$ is defined by
$\phi^{(n)}\left([x_{ij}]_{i,j}\right)=[\phi(x_{ij})]_{i,j}$; (iii) for any operator systems $\oss$ and $\ost$, 
$\oss\otimes\ost$ shall denote their algebraic tensor product.

 \subsection{Quotients of operator systems}

 Assume that $\oss$ is an operator system and that $\kj\subset\oss$ is norm-closed $*$-subspace that does not
 contain $1_\oss$. Let $q:\oss\rightarrow\oss/\kj$
 denote the canonical linear map of $\oss$ onto the vector space $\oss/\kj$. The vector space 
 $\oss/\kj$ has an induced involution
 defined by $q(X)^*=q(X^*)$. 
 
 We will denote elements of $\oss$ by uppercase letters, such as $X$, and elements of $\oss/\kj$
 by lowercase letters, as in $x=q(X)$, or by ``dots'' as in $\dot{X}=q(X)$.
  
 For each $n\in\mathbb N$, 
 there is a linear isomorphism
 \[
 (\oss/\kj)\otimes \mn\,\cong\,(\oss\otimes\mn)/(\kj\otimes \mn)\,.
 \]
 The canonical linear surjection $\oss\otimes\mn\rightarrow (\oss\otimes\mn)/(\kj\otimes \mn)$ is denoted by $q^{(n)}$ and we shall view $q^{(n)}(X)$
 as an $n\times n$ matrix $x=\dot{X}$ of $\dot{X}_{ij}\in\oss/\kj$, thereby identifying $(\oss/\kj)\otimes\mn$ with $\mn(\oss/\kj)$.
 Further, each
matrix space $\mn(\oss/\kj)$ is an involutive vector space under the involution $h=[h_{ij}]_{1\leq i,j\leq n}
 \mapsto {h}^*=[{h}_{ji}^*]_{1\leq i,j\leq n}$. The real vector space of hermitian (or selfadjoint elements) is denoted by
 $ \mn(\oss/\kj)_{\rm sa}$. 
 
 Consider the subset
$D_n(\oss/\kj)$ of $\mn(\oss/\kj)_{\rm sa}$ defined by
 \[
 D_n(\oss/\kj)\;=\;\{\dot{H}\,:\,\exists\,K\in  \mn(\kj)_{\rm sa}\;\mbox{such that }  H+K \in\mn(\oss)_+\}\,.
 \]
 The collection $\{D_n(\oss/\kj)\}_{n\in\mathbb N}$ is a family of cones 
 that endow $\oss/\kj$ with the structure of a matricially ordered space.
 However, this ordering will not be an operator system on $\oss/\kj$ in general.
 
 \begin{definition}
 A subspace $\kj\subset\oss$ is a \emph{kernel} if there are
 an operator system $\ost$ and a completely positive linear map $\phi:\oss\rightarrow\ost$
 such that $\kj=\ker\phi$. 
 \end{definition}
 
 If $\kj\subset\oss$ is a kernel, then 
 define a subset $C_n(\oss/\kj)\subset  \mn(\oss/\kj)_{\rm sa}$ by
 \[
 C_n(\oss/\kj)\,=\, \{\dot{H}\,:\,\forall\,\varepsilon>0\;\exists\,K_\varepsilon\in \mn(\kj)_{\rm sa}
 \; \mbox{such that }\varepsilon 1+H+K_\varepsilon\in\mn(\oss)_+\}\,.
  \]
That is,
\[
C_n(\oss/\kj)\;=\;\{\dot{H}\,:\, \varepsilon\dot{1}+\dot{H}\in  D_n(\oss/\kj),\,\forall\,\varepsilon>0\}\,.
 \]
 The collection $\{C_n(\oss/\kj)\}_{n\in\mathbb N}$ is a family of cones that endow $\oss/\kj$ with the structure of an operator system
 with (Archimedean) order unit $\dot{1}=q(1)$. 
 
 \begin{definition}
The operator system $\left( \oss/\kj, \{C_n(\oss/\kj)\}_{n\in\mathbb N}, q(1)\right)$ 
 is called a \emph{quotient operator system}.
 \end{definition}
 
 The first basic result concerning quotient operator systems is the 
 First Isomorphism Theorem \cite[Proposition 3.6]{kavruk--paulsen--todorov--tomforde2010}.
 
 \begin{theorem}\label{1st iso thm} {\rm (First Isomorphism Theorem)} If $\phi:\oss\rightarrow\ost$ is a nonzero
 completely positive linear map, then there is a unique completely positive linear map $\dot{\phi}:\oss/\ker\phi\rightarrow\ost$
 such that $\phi=\dot{\phi}\circ q$.
 \end{theorem}

 Associated with $D_n(\oss/\kj)$ and $C_n(\oss/\kj)$ are 
 the subsets $\mathcal D_n(\oss/\kj)$ and $\mathcal C_n(\oss/\kj)$ of $(\oss\otimes\mn)_{\rm sa}$
 defined by
 \[
 \mathcal D_n(\oss/\kj)\;=\;q_n^{-1}\left(D_n(\oss/\kj) \right) \quad\mbox{and}\quad
\mathcal C_n(\oss/\kj)\;=\; q_n^{-1}\left( C_n(\oss/\kj)\right)\,.
\]

 Clearly $D_n(\oss/\kj)\subseteq C_n(\oss/\kj)$. Examples in which the inclusion $D_n(\oss/\kj)\subseteq C_n(\oss/\kj)$ is proper are given in 
 \cite{kavruk--paulsen--todorov--tomforde2010}.

\begin{definition}
A kernel $\kj\subset\oss$ is \emph{completely order proximinal} if 
 $D_n(\oss/\kj)=C_n(\oss/\kj)$, for every $n\in\mathbb N$. 
\end{definition}

This notion of a completely order proximinal kernel will be of importance for the quotient operator
systems we study in later sections.

 \begin{example} If $\csta$ is a unital C$^*$-algebra and if $\kj\subset \csta$ is the kernel of a $*$-homomorphism, then 
 $\kj$ is completely order proximinal and
 \[
\left( (\csta/\kj)\otimes\mn\right)_+\,=\,D_n(\csta/\kj)\,=\,C_n(\csta/\kj)\,,\;\mbox{for all } n\in\mathbb N\,.
\]
 \end{example}
 
 \begin{proof} Let $\kj=\ker\pi$, where $\pi$ is a $*$-homomorphism which, without loss of generality, we assume to be unital. 
 Note that $D_n(\csta/\kj)$ is precisely the positive cone $(\csta\otimes\mn)_{+}$
 of the quotient C$^*$-algebra $\csta/\kj$.
 Because any positive element $h=\dot{H}$ of
 $(\csta/\kj)\otimes\mn \cong (\csta\otimes\mn)/(\kj\otimes\mn)$
 lifts to an element of the form $H+K$, where $K\in \kj\otimes \mn$ and $H\in(\csta\otimes\mn)_+$,
 \[
 \mathcal D_n(\csta/\kj)\;=\;\{H+K\,:\,H\in(\csta\otimes\mn)_{+}, \;K\in (\kj\otimes\mn)_{\rm sa}\}\,.
 \]
 Suppose now that $H\in\mathcal C_n(\csta/\kj)$. Then, for every $\varepsilon>0$, 
 $\varepsilon 1+H\in \mathcal D_n(\csta/\kj)$ and so every $\varepsilon q_n(1)+q_n(H)=\varepsilon\dot{1}+h$ is
 positive in $(\csta/\kj)\otimes\mn$. But as $\dot{1}\in\csta/\kj$ is an Archimedean order unit for the 
 quotient C$^*$-algebra $\csta/\kj$, $\varepsilon \dot{1}+h$ is positive for all $\varepsilon>0$ only if $h$ is positive,
 which is to say that $H\in  \mathcal D_n(\csta/\kj)$.
 \end{proof}

The proof above shows that, in the case of C$^*$-agebras $\csta$ and ideals $\kj\subset\csta$, $ \mathcal D_n(\csta/\kj)=
\mathcal C_n(\csta/\kj)$ for every $n\in\mathbb N$. This in fact leads to the following useful
criterion for completely order proximinal kernels in operator systems.

\begin{proposition}\label{proxchar} Let $\kj$ be a kernel in an operator system $\oss$. Then 
\begin{enumerate}
\item $\mathcal D_n(\oss/\kj) = \mn(\kj)+\mn(\oss)_+$,
\item $\mathcal C_n(\oss/\kj)$ is the norm closure of $\mn(\kj)+\mn(\oss)_+$, and
\item $\kj$ is completely order proximinal if and only if $\mn(\kj)+\mn(\oss)_+$ 
is closed for every $n \in \mathbb N$.
\end{enumerate}
\end{proposition}
\begin{proof}
The first statement is obvious. To see the second statement, first note that
$\mathcal C_n(\oss/\kj) = \{ H : \forall \epsilon > 0, \epsilon 1 + H \in \mathcal
D_n(\oss/\kj) \}= \{ H : \forall \epsilon > 0,  \epsilon 1 + H \in
\mn(\oss)_+ + \mn(\kj)  \}.$ Thus, if $H \in \mathcal
C_n(\oss/\kj)$ then for every $\epsilon > 0$ there exist
$P_{\epsilon} \in \mn(\oss)_+$ and $K_{\epsilon} \in \mn(\kj),$ such
that
$\epsilon 1 + H = P_{\epsilon} + K_{\epsilon}$ and it follows that $H$
is in the norm closure of $\mn(\oss)_+ + \mn(\kj)$.
 
For the converse it is sufficient to assume that $H=H^*$ is in the
closure of $\mn(\oss)_+ + \mn(\kj)_{\rm sa}$ and prove that $\epsilon 1 + H \in
\mn(\oss)_+ + \mn(\kj)$ for every $\epsilon > 0$.
Since $H$ is in the closure and $\mn(\kj)^* =\mn(\kj)$, for every
$\epsilon > 0$ we may choose $P_{\epsilon} \in \mn(\oss)_+$ and
$K_{\epsilon} \in \mn(\kj)_{\rm sa}$ such that $\|H- (P_{\epsilon} +
K_{\epsilon}) \| < \epsilon.$ This implies that $Q_{\epsilon} =
\epsilon 1 + H - P_{\epsilon} - K_{\epsilon} \in \mn(\oss)_+$, and
hence
$\epsilon 1 + H = (Q_{\epsilon} + P_{\epsilon}) + K_{\epsilon} \in
\mn(\oss)_+ + \mn(\kj)$.  This proves the second statement.

The third statement follows by combining the first and second statements.
\end{proof}

\subsection{Definitions: complete quotient maps and complete order injections}

\begin{definition} 
A linear map $\phi:\oss\rightarrow\ost$
of operator systems is a \emph{complete order isomorphism} if
\begin{enumerate}
\item $\phi$ is a linear isomorphism and
\item $\phi$ and $\phi^{-1}$ are completely positive.
\end{enumerate}
\end{definition}

\begin{definition} 
A linear map $\phi:\oss\rightarrow\ost$
of operator systems is a \emph{complete order injection}, or a \emph{coi map}, if
$\phi$ is a complete order isomorphism between $\oss$ and the operator system $\phi(\oss)$
with order unit $\phi(1_\oss)$.
\end{definition}

\begin{definition} 
A completely positive linear map $\phi:\oss\rightarrow\ost$
of operator systems is a \emph{complete quotient map} if 
the uniquely determined completely positive linear 
map $\dot{\phi}:\oss/\ker\phi\rightarrow\ost$ is a complete order isomorphism.
\end{definition}

A basic fact about complete quotient maps that we shall use repeatedly is:

\begin{proposition}\label{lifting positives} If $\phi:\oss\rightarrow\ost$ is a complete quotient map and if
$\ker\phi$ is completely order proximinal,
then $\phi^{(n)}$ maps $\mn(\oss)_+$ onto $\mn(\ost)_+$, for every $n\in\mathbb N$.
\end{proposition}
  
The next fact is also basic.

\begin{proposition}\label{quo lemma} If $\ost$ is an operator subsystem of $\oss$, and if $\kj$ is a kernel such that
$\kj\subset\ost\subseteq\oss$, then there is a complete order embedding $\ost/\kj\rightarrow\oss/\kj$.
\end{proposition}

 \subsection{Tensor products of operator systems}
 
 \subsubsection{The minimal tensor product}
 
 The \emph{matricial state space} of an operator system $\oss$ is the set 
 \[
 {\rm S}_\infty(\oss)=\displaystyle\bigcup_{n\in\mathbb N} {\rm S}_n(\oss)\,,\;\mbox{ where }\;
 {\rm S}_n(\oss)=\{\mbox{all ucp maps }\oss\rightarrow\mn\}\,.
 \] 
 For every $p\in\mathbb N$, let 
 \[
 \begin{array}{rcl}
 \cminp(\oss,\ost)&=&\{x\in\mpee(\oss\otimes\ost)\,:\,(\phi\otimes\psi)^{(p)}(x)\in \mpee(\mk\otimes\mm)_+,
 \\ && \qquad\qquad\;\forall (\phi,\psi)\in 
  {\rm S}_k(\oss)\times {\rm S}_m(\ost),\;\forall\,(k,m)\in\mathbb N\times\mathbb N\}\,.
\end{array}
  \]
 The collection $\{\cminp(\oss,\ost)\}_{p\in\mathbb N}$ induces an operator system
 structure on $\oss\otimes\ost$ with order unit $1_\oss\otimes 1_\ost$. The operator system that arises from this matricial order structure
 is denoted by $\oss\omin\ost$ and is called
 the \emph{minimal tensor product} of $\oss$ and $\ost$.  
 
\subsubsection{The maximal tensor product}
Let $\dmaxp(\oss,\ost)\subset\mpee(\oss\otimes\ost)$ denote the subset of all elements of the form
 \[
\alpha(s\otimes t)\alpha^*\,,\;\mbox{where }\, s\in\mk(\oss)_+\,,\;t\in\mm (\ost)_+\,,
 \;\alpha:\mathbb C^k\otimes\mathbb C^m\rightarrow
 \mathbb C^p\;\mbox{is linear}\,.
 \]
 For every $p\in\mathbb N$, let
 \[
 \cmaxp(\oss,\ost)\;=\;\{h\in\mpee(\oss\otimes\ost)\,:\,\varepsilon1_{\mpee}+h\in\dmaxp(\oss,\ost)\;\mbox{ for all }\varepsilon>0\}\,.
 \]
 The collection $\{\cmaxp(\oss,\ost)\}_{p\in\mathbb N}$ induces an operator system
 structure on $\oss\otimes\ost$ with order unit $1_\oss\otimes 1_\ost$. The operator system that arises from this matricial order structure
 is denoted by $\oss\omax\ost$ and is called the \emph{maximal tensor product} of $\oss$ and $\ost$. 
  
 \begin{proposition}\label{max quo} If $\phi:\oss\rightarrow\ost$ is a complete quotient map, then $\phi\otimes{\rm id}_\osr:
 \oss\omax\osr\rightarrow\ost\omax\osr$ is a complete quotient map for every operator system $\osr$.
 \end{proposition}
 
 \begin{proof} Let $\kj=\ker\phi$. Because $\dot{\phi}:\oss/\kj\rightarrow\ost$ is a complete order isomorphism, 
 so is $\dot{\phi}\otimes{\rm id}_\osr:(\oss/\kj)\omax\osr\rightarrow\ost\omax\osr$. Therefore, $\phi\otimes{\rm id}_\osr$
 is a complete quotient map if $(\oss\omax\osr)/\ker(\phi\otimes{\rm id}_\osr)$ and $(\oss/\kj)\omax\osr$ are completely order isomorphic.
 
 Let $\tilde q:\oss\omax\osr\rightarrow(\oss\omax\osr)/\ker(\phi\otimes{\rm id}_\osr)$
 be the canonical surjective cp map and define a bilinear map
 \[
 \sigma: (\oss/\kj)\times\osr\rightarrow (\oss\omax\osr)/\ker(\phi\otimes{\rm id}_\osr)
 \]
 by $\sigma(x,y)\;=\tilde q(X\otimes y)$, where $X\in\oss$ is any element for which $q(X)=x$. Note that $\sigma$ is well defined.
 Furthermore, $\sigma$ is jointly completely positive and ``unital,'' and so by the universal property of the max tensor product 
 \cite[Theorem 5.8]{kavruk--paulsen--todorov--tomforde2009}, there is a ucp extension 
 $\sigma:(\oss/\kj)\omax\osr\rightarrow (\oss\omax\osr)/\ker(\phi\otimes{\rm id}_\osr)$. We
 claim that $\sigma$ is a complete order isomorphism.
 
 By linear algebra, the restriction of $\sigma$ to the algebraic tensor product is a linear isomorphism between $(\oss/\kj)\otimes\osr$
 and $(\oss\otimes\osr)/\ker(\phi\otimes{\rm id}_\osr)$, and therefore $\sigma$ is a linear isomorphism
 of $(\oss/\kj)\omax\osr$ and $(\oss\omax\osr)/\ker(\phi\otimes{\rm id}_\osr)$.
 
To show that $\sigma^{-1}$ is completely positive, select a positive element $b$ in the $p\times p$ matrices over
$(\oss\omax\osr)/\ker(\phi\otimes{\rm id}_\osr)$. Hence, for every $\varepsilon>0$ there is are $P\in\mk(\oss)_+$, $Q\in \mm(\osr)_+$
and $\alpha:\mathbb C^k\otimes\mathbb C^m\rightarrow\mathbb C^p$ linear such that $\varepsilon\dot{1}+b=\alpha(q(P)\otimes Q)\alpha^*$.
Hence, $\varepsilon\sigma^{-1}{}^{(p)}(\dot{1})+ \sigma^{-1}{}^{(p)}(b)\in \dmaxp(\oss/\kj,\osr)$ for all $\varepsilon>0$. Because $\sigma^{-1}$
is unital, this implies that $\sigma^{-1}{}^{(p)}(b)\in\mpee( (\oss/\kj)\omax\osr)_+$.
\end{proof}
 
 \begin{corollary}\label{max quo cor} If $\phi:\oss\rightarrow\ost$ is a complete quotient map, then $\phi\otimes\phi:
 \oss\omax\oss\rightarrow\ost\omax\ost$ is a complete quotient map.
 \end{corollary}
 
 \begin{proof} Because $\phi\otimes\phi=(\phi\otimes{\rm id}_\ost) \circ ({\rm id}_\ost\otimes\phi)$
 and a composition of complete quotient maps is a complete quotient map, the result follows.
 \end{proof}
 
 \subsubsection{The commuting tensor product}
 
 Let ${\rm cp}(\oss,\ost)$ denote the set of all pairs $(\phi,\psi)$ of completely positive linear maps that map $\oss$ and $\ost$, respectively, into
 a common $\bh$ and such that $\phi(x)\psi(y)=\psi(y)\phi(x)$, for all $(x,y)\in\oss\times\ost$. Such a pair determines a bilinear map
 $\phi\cdot\psi:\oss\times\ost\rightarrow\bh$ via $\phi\cdot\psi(x,y)=\phi(x)\psi(y)$, and so there exists a unique linear map, denoted again by $\phi\cdot\psi$,
 of $\oss\otimes\ost$ into $\bh$ and for which $\phi\cdot\psi(x\otimes y)=\phi(x)\psi(y)$, for all elementary tensors $x\otimes y\in\oss\otimes\ost$.
 Let
 \[
 \begin{array}{rcl}
 \ccomp(\oss,\ost) & =& \{x\in\mpee(\oss\otimes\ost)\,:\,(\phi\cdot\psi)^{(p)}(x)\in \mpee(\bh)_+, \\
&&\qquad\qquad\qquad\qquad\;\forall \,(\phi,\psi)\in {\rm cp}(\oss,\ost) \}\,.
\end{array}
\]
The collection $\{\ccomp(\oss,\ost)\}_{p\in\mathbb N}$ induces an operator system
 structure on $\oss\otimes\ost$ with order unit $1_\oss\otimes 1_\ost$. The operator system that arises from this matricial order structure
 is denoted by $\oss\oc\ost$ and is called the \emph{commuting tensor product} of $\oss$ and $\ost$.

\subsubsection{Properties of tensor products}
 
 Evidently,
 \[
 \cmaxp(\oss,\ost)\subseteq \ccomp(\oss,\ost)\subseteq\cminp(\oss,\ost)\,,
 \]
 and so the maps 
$\oss\omax\ost\rightarrow\oss\oc\ost$ and $\oss\oc\ost\rightarrow\oss\omin\ost$ arising from the
identity map of $\oss\otimes\ost$ are ucp; we denote
 this by
 \[
 \oss\omax\ost  \subseteq \oss\oc\ost\subseteq  \oss\omin\ost\,.
 \]
 
  \begin{definition} Assume that $\oss$ is an operator system.
  \begin{enumerate}
 \item $\oss$ is said to be \emph{nuclear} or  \emph{{\rm (min,max)}-nuclear}
if $\oss\omin\ost=\oss\omax\ost$ for every 
 operator system $\ost$.
 \item $\oss$ is \emph{{\rm (min,c)}-nuclear} if $\oss\omin\ost=\oss\oc\ost$ for every 
 operator system $\ost$.
\end{enumerate}
 \end{definition}
 
 All unital C$^*$-algebras that are nuclear in the sense of C$^*$-algebraic tensor products are also (min,max)-nuclear
 \cite[Proposition 5.15]{kavruk--paulsen--todorov--tomforde2009}.
 Although finite-dimensional C$^*$-algebras are, therefore,  nuclear, it is not the case that every finite-dimesional operator 
 system is nuclear \cite[Theorem 5.18]{kavruk--paulsen--todorov--tomforde2009}.
 
 By definition, an operator system $\oss$ is endowed with a specific matricial ordering in which the positive cone in the vector space $\mpee(\oss)$
 of $p\times p$ matrices is denoted by $\mpee(\oss)_+$. Expressed as tensor products, we have 
 \[
\mn( \oss)_+\,=\,(\oss\omax\mn)_+\,=\,(\oss\oc\mn)_+\,=\,(\oss\omin\mn)_+\,.
\]

\subsection{Duals of finite-dimensional operator systems}

For any finite-dimensional operator system $\oss$, let $\oss^d$ denote the dual of $\oss$. For every $n\in\mathbb N$, there
are the following natural identifications:
\[
\mn(\oss^d)\;\cong \mathcal L(\oss^{dd},\mn)\;=\;\mathcal L(\oss,\mn)\;\cong\;\mn(\oss)^d\,.
\]
In the first of these identifications, and using $\leftrightarrow$ to denote ``is identified with,''
\[
G=[g_{ij}]_{i,j}\in \mn(\oss^d)\;\longleftrightarrow\;\hat G\in\mathcal L(\oss,\mn),\;\hat G(s)=[g_{ij}(s)]_{i,j}\,.
\]
In the second of the identifications,
\[
\hat G\in\mathcal L(\oss,\mn)\; \longleftrightarrow\; \varphi_G\in\mn(\oss)^d,\;\varphi_G\left([s_{ij}]_{i,j}\right)\,=\,\sum_{i,j}g_{ij}(s_{ij})\,.
\]
A linear functional on $\mn(\oss)$ of the form $\varphi_G$ satisfies $\varphi_G(x)\geq0$, for all $x\in\mn(\oss)_+$, if and only if $\hat G:\oss\rightarrow\mn$
is completely positive. Hence, $\mn(\oss^d)_+$ is defined to be:
\[
\begin{array}{rcl}
\mn(\oss^d)_+ &=& \{ G\in \mn(\oss^d)\,:\,\hat G:\oss\rightarrow \mn \;\mbox{is completely positive}\} \\
                             &=& \{G\in \mn(\oss^d)\,:\,\varphi_G(x)\geq0, \;\forall\,x\in\mn(\oss)_+\}\,.
\end{array}
\]
One can define these cones for any operator system $\oss$ and they endow $\oss^d$ with a matricial ordering. However, because $\oss$
has finite dimension, this matricial ordering on $\oss^d$ gives rise to an operator system structure on $\oss^d$ in which an Archimedean
order unit is given by $e=\varphi$, for some faithful state $\varphi$ of $\oss$ \cite[\S4]{choi--effros1977}.  

For every linear map $\phi:\oss\rightarrow\ost$, let $\phi^d$ denote the adjoint of $\phi$ as a linear map $\ost^d\rightarrow\oss^d$.
That is, $\phi^d(\psi)[s]=\psi(\phi(s))$, for all $\psi\in\ost^d$, $s\in\oss$.                            

\begin{proposition}\label{q-coi} The following statements are equivalent for a linear map $\phi:\oss\rightarrow\ost$ of
finite-dimensional operator systems:
\begin{enumerate}
\item\label{q-coi-1} $\phi$ is a complete quotient map;
\item\label{q-coi-2} $\phi^d$ is a complete order injection.
\end{enumerate}
\end{proposition}

\begin{proof} Assume that $\phi$ is a complete quotient map and fix $n\in\mathbb N$. 
Recall that $\phi^d$ is completely positive 
since, for every $G\in \mn(\ost^d)$,
$\phi^d{}^{(n)}(G)$ is identified with $\hat G\circ\phi$, a composition of
completely positive linear maps; thus, $\phi^d{}^{(n)}(G)\in\mn(\oss^d)_+$. 

To show that $\phi^d$ is a complete order injection, it is sufficient to show that if 
$\phi^d{}^{(n)}(G)\in\mn(\oss^d)_+$, for some $G\in\mn(\ost^d)$, then necessarily $G\in\mn(\ost^d)_+$.
Therefore, it is necessary to show that $\hat G$ is completely positive. Let $p\in\mathbb N$ and
$y\in\mpee(\ost)_+$. By hypothesis, $\phi$ is a complete quotient map; thus, there is a $x\in\mpee(\oss)_+$ such
that $\phi^{(p)}(x)=y$ (Proposition \ref{lifting positives}). Thus, $\hat G^{(p)}(y)=\hat G^{(p)}\circ\phi^{(p)}(x)\in
\mpee(\mn)_+$, which shows that $G\in\mn(\ost^d)_+$. This proves that $\phi^d$ is a complete order injection.

Conversely, assume that $\phi^d$ is a complete order injection.  Let $\ose$ denote an arbitrary operator system.
By way of the identification of $\mn(\ose^d)_+$
with the linear functionals $\varphi$ on $\mn(\ose)$ for which $\varphi(x)\geq0$ for all $x\in\mn(\ose)_+$, we identify
$\mn(\ose^d)_+$ with the dual cone $\left[ \mn(\ose)_+\right]^{\#}$ of the cone $\mn(\ose)_+$:
\[
\begin{array}{rcl}
\mn(\ose^d)_+&=&\left[ \mn(\ose)_+\right]^{\#}\\  
                       &=&\{\psi\in\mn(\ose)^d\,:\,\psi(x)\geq0,\;\forall\,x\in\mn(\ose)_+\}\,.
\end{array}
\]
By hypothesis, $\psi\in\mn(\ost^d)_+$ if and only if $\phi^d{}^{(n)}(\psi)\in \mn(\oss^d)_+$. The dual cone of the 
cone $\phi^d{}^{(n)}\left(\mn(\oss)_+\right)$ in $\mn(\ost)$ is given by
\[
\begin{array}{rcl}
\left[\phi^d{}^{(n)}\left(\mn(\oss)_+\right)\right]^{\#}
&=&\{\psi\in\mn(\ost)^d\,:\,\psi(y)\geq0,;\forall\,y\in \phi^d{}^{(n)}\left(\mn(\oss)_+\right) \} \\
&=&\{\psi\in\mn(\ost)^d\,:\,\psi\circ \phi^{(n)}(x)\geq0,\;\forall\,x\in \mn(\oss)_+\} \\
&=&\{\psi\in\mn(\ost)^d\,:\,  \phi^d{}^{(n)}(\psi) \in \mn(\oss^d)_+\} \\
&=&\left[\mn(\ost)_+\right]^{\#}\quad\mbox{\rm (as $\phi$ is coi)}\,.
\end{array}
\]
Hence, by the Bidual Theorem,
\[
\phi^d{}^{(n)}\left(\mn(\oss)_+\right)\,=\,\left[\phi^d{}^{(n)}\left(\mn(\oss)_+\right)\right]^{\#\#}
\,=\,\left[\mn(\ost)_+\right]^{\#\#}\,=\,\mn(\ost)_+\,,
\]
which implies that $\dot{\phi}$ is a complete order isomorphism.
\end{proof}

\begin{remark} The use of finite-dimensional operator systems in Proposition \ref{q-coi}
is not essential, as the same arguments can be applied to the case of arbitrary $\oss$ and
$\ost$. What is lost is the fact that the duals $\oss^d$ and $\ost^d$ are operator systems. However,
duals of operator systems are matricially ordered spaces and one can speak of completely positive maps
between such spaces. Thus, Proposition \ref{q-coi} holds in the category whose objects are
matricially normed spaces
and whose morphisms are completely positive linear maps.
\end{remark}

\begin{proposition}\label{fd dual tensor coi} If $\oss$ and $\ost$ are finite-dimensional operator systems, then the operator systems
$\oss^d\omax\ost^d$ and $\left(\oss\omin\ost\right)^d$ are completely order isomorphic.
\end{proposition} 

\begin{proof} Let $\delta$ and $\zeta$ be faithful states on $\oss$ and $\ost$, respectively,
so that $e_{\oss^d}=\delta$ and $e_{\ost^d}=\zeta$ are the Archimedean order units for the
operator systems $\oss^d$ and $\ost^d$.

We first show that there is a completely positive embedding of $\iota_1:\oss^d\omax\ost^d\rightarrow \left(\oss\omin\ost\right)^d$.
Let $P\in\mk(\oss^d)_+$ and $Q\in\mm(\ost^d)_+$; we aim to show that $\hat P\otimes \hat Q$ is a completely
positive map $\oss\omin\ost\rightarrow\mk\otimes\mm$. Any completely positive map of an operator system into
a matrix algebra is a matricial convex combination of matricial states; thus,
\[
\hat P\,=\,\sum_{\mu}\alpha_\mu^*\phi_\mu\alpha_\mu\,,\quad
\hat Q\,=\,\sum_{\nu}\beta_\nu^*\psi_\nu\beta_\nu\,,
\]
where $\phi_\mu$ and $\psi_\nu$ are matricial states on $\oss$ and $\ost$, respectively, and $\alpha_\mu$ and $\beta_\nu$
are rectangular matrices for which $\sum_\nu\alpha_\mu^*\alpha_\mu=1_k$, $\sum_\nu\beta_\nu^*\beta_\nu=1_m$. Hence,
\[
\hat P \otimes \hat Q\,=\,\sum_{\mu,\nu} (\alpha_\mu\otimes\beta_\mu)^*(\phi_\mu\otimes\phi_\nu)(\alpha_\mu\otimes\beta_\nu)
\]
is completely positive. Hence, every $\alpha(\hat P\otimes \hat Q)\alpha^*$ is completely positive, for every linear map 
$\alpha:\mathbb C^k\otimes\mathbb C^m\rightarrow\mathbb C^n$. Identifying $\hat P\otimes\hat Q$ and $P\otimes Q$, 
this proves that 
\[
\alpha(P\otimes Q)\alpha^*\in\mn\left((\oss\omin\ost)^d\right)_+\,.
\]
Therefore, if $G\in \mn(\oss^d\omax\ost^d)_+$, then for every $\varepsilon>0$ there are $P\in\mk(\oss^d)_+$, $Q\in\mm(\ost^d)_+$,
and $\alpha:\mathbb C^k\otimes\mathbb C^m\rightarrow\mathbb C^n$ linear such that 
\[
G+\varepsilon(e_{\oss^d}\otimes e_{\ost^d})\,=\,\alpha(P\otimes Q)\alpha^*\in\mn\left((\oss\omin\ost)^d\right)_+\,.
\]
Because $G+\varepsilon(e_{\oss^d}\otimes e_{\ost^d})\in \mn\left((\oss\omin\ost)^d\right)_+$ for every $\varepsilon>0$, 
the Archimedean property
implies that $G\in\mn\left((\oss\omin\ost)^d\right)_+$. Hence, the canonical linear embedding 
$\iota_1:(\oss^d\otimes\ost^d)\rightarrow(\oss\omin\ost)^d$ is completely positive.

We now show that there is a completely positive embedding of $(\oss^d\omax\ost^d)^d$ into $\oss\omin\ost$.
To this end, fix $n$ and let $G\in \mn\left((\oss^d\omax\ost^d)^d\right)_+$. Because $\oss^{dd}=\oss$ and $\ost^{dd}=\ost$,
$G=[g_{ij}]_{i,j}$ for some $g_{ij}\in\oss\otimes\ost$. Having identified $G$ as an element of $\mn(\oss\otimes\ost)$, we need
only show that $G\in\mn(\oss\omin\ost)_+$. To do so, let $\psi:\oss\rightarrow\mk$ and $\psi:\ost\rightarrow\mm$ be any
pair of matricial states. Thus, $\phi\in\mk(\oss^d)_+$ and $\psi\in\mm(\ost^d)_+$. 
By hypothesis, $\hat G:\oss^d\omax\ost^d\rightarrow\mn$ is completely positive, and so $\hat G^{(km)}(\phi\otimes\psi)$ is 
positive in $\mn(\mk\otimes\mm)$. In writing $\phi$ and $\psi$ as matrices of linear functionals $\phi_{k\ell}$, $\psi_{\mu\nu}$,
we note that
\[
\hat G^{(km)}(\phi\otimes\psi)\,=\,[G(\phi_{k\ell}\otimes\psi_{\mu\nu})]\,=\,[\phi_{k\ell}\otimes\psi_{\mu\nu})(g_{ij})]
\,=\,\phi\otimes\psi\,{}^{(n)}(G)\,,
\]
which implies that $\phi\otimes\psi\,{}^{(n)}(G)\in \mn(\mk\otimes\mm)_+$. Therefore, the canonical linear embedding
$\iota:(\oss^d\omax\ost^d)^d\rightarrow\oss\omin\ost$ is completely positive.

To complete the proof, the completely positive embedding $\iota:(\oss^d\omax\ost^d)^d\rightarrow\oss\omin\ost$ leads to a
completely positive embedding 
\[
\iota_2=\iota^d:(\oss\omin\ost)^d\rightarrow (\oss^d\omax\ost^d)^{dd}\,=\,\oss^d\omax\ost^d\,.
\]
In combination with the embedding $\iota_1$ and because $\iota_1=\iota_2^{-1}$, we obtain a complete order isomorphism of 
$\oss^d\omax\ost^d$ and $\left(\oss\omin\ost\right)^d$.
\end{proof}

\begin{remark} Since $\oss^d$ is an operator system with Archimedian order 
unit $\delta$ and $\ost^d$ is an operator system with Archimedian order unit $\zeta,$ 
the image of $\delta\otimes\zeta$ under the above complete order isomorphism is an 
Archimedean order unit for $(\oss\omin\ost)^d.$ Thus, in particular, $\delta\otimes\zeta$ 
is a faithful state on $\oss\omin\ost$.
\end{remark}

\section{The Operator System $\mn/\kj_n$}\label{S:osw}
  
Assume that $n\geq 2$ and that $w_2,\dots,w_n$ are $n-1$ universal unitaries 
that generate the  full group C$^*$-algebra $\cstar(\fnl)$, where $\fnl$
is the free group on $n-1$ generators. Let $w_1=1\in\cstar(\fnl)$.
Throughout, ${\rm tr}$ denotes the standard (non-normalised) trace functional on a full matrix algebra.
  
\begin{definition} For each $n\geq 2$, let
\begin{enumerate}
\item $\kj_n\subset\mn$ be the vector subspace of all diagonal matrices 
$D\in\mn$ with $\mbox{\rm tr}(D)=0$, and
\item $\osw_n$ be the
operator system in $\cstar(\fnl)$ spanned by $\{w_iw_j^*\,:\,1\leq i,j\leq n\}$.
\end{enumerate}
\end{definition}

We first show that $\kj_n$ is a kernel. To do so, consider the unital linear 
map $\phi:\mn\rightarrow\cstar(\fnl)$ defined
on the matrix units of $\mn$ by 
\begin{equation}\label{defn of phi}
\phi(E_{ij})=\frac{1}{n}w_iw_j^*\,,\;1\leq i,j\leq n\,.
\end{equation} The Choi matrix
corresponding to $\phi$ is
\[
\left[\phi(E_{ij})\right]_{1\leq i,j\leq n}\;=\;\frac{1}{n}[w_iw_j^*]_{1\leq i,j \leq n}\;=\;\frac{1}{n}W^*YW\in (\cstar(\fnl)\otimes\mn)_+\,,
\]
where $W=\sum_iw_i^*\otimes E_{ii}$ and $Y=\sum_{i,j}1\otimes E_{ij}\,\in\, (\cstar(\fnl)\otimes\mn)_+$.
Hence, $\phi$ is a completely positive linear map of $\mn$ onto $\osw_n$.
It is clear that $\kj_n\subseteq\ker\phi$. Conversely, suppose that $A\in\ker\phi$. Then,
\[
0\;=\;\phi(A)\;=\;(\sum_{i=1}^n a_{ii})1\,+\,\frac{1}{n}\sum_{j\neq i}a_{ij}(w_iw_j^*)\,,
\]
and so $a_{ij}=0$ for all $j\neq i$ and $\mbox{tr}(A)=0$. That is, $\ker\phi\subseteq\kj_n$. 

Because $\kj_n$ is a kernel, we form and study the quotient operator system $\mn/\kj_n$. In this section we will show that
the ucp map $\phi:\mn\rightarrow\osw_n$ is a complete quotient map and that the C$^*$-envelope of the operator system
$\mn/\kj_n$ is $\cstar(\fnl)$.

\begin{notation} 
The order unit $q(1)$ of $\mn/\kj_n$ is denoted  by $\dot{1}$ and $e_{ij}$ denotes $q(E_{ij})$ for
every matrix unit $E_{ij}$ of $\mn$.
\end{notation} 
  
\begin{lemma}\label{norm} For every $i$ and $j$, $e_{ii}=\frac{1}{n}\dot{1}$ and $\|e_{ij}\|=\frac{1}{n}$.
\end{lemma}

\begin{proof} 
Suppose that $j\neq i$. The matrix $Q=\left[\begin{array}{cc} E_{ii} & E_{ij} \\ E_{ji} & E_{jj}\end{array}\right]$ is positive in
$\mtwo(\mn)$, and therefore $q^{(2)}(Q)$ is positive in $\mtwo(\mn/\kj_n)$. However, the matrix
\[
q^{(2)}(Q) \;=\; \left[ \begin{array}{cc} \frac{1}{n}\dot{1} & e_{ij} \\ e_{ji} & \frac{1}{n}\dot{1}\end{array}\right]  
\]
is positive only if $\|e_{ij}\|\leq\frac{1}{n}$. Consider now the density matrix $\rho\in\mn$ with $\frac{1}{n}$ in every entry
and let $s_\rho$ be the state on $\mn$ defined by $s_\rho(X)=\mbox{tr}(\rho X)$. Then $s_\rho(\kj_n)=\{0\}$ and so we obtain
a well-defined state $\dot{s}$ on $\mn/\kj_n$ via $\dot{s}(\dot{X})=s_\rho(X)$. With this state  $\dot{s}$, we have
$\dot{s}(e_{ij})=\mbox{tr}(\rho E_{ij})=\frac{1}{n}$, which implies that $\|e_{ij}\|\geq\frac{1}{n}$.
\end{proof}

 The result above shows that the norm on the operator system quotient $\mn/\kj_n$
 is quite different from the {\em quotient norm} of $\mn$ by
 the subspace $\kj_n$. Indeed, in the usual quotient norm, one has
 $\|e_{ij}\|_q \equiv \inf \{ \|E_{ij} + K \|: K \in \kj_n \} = 1$,
 when $i \ne j$.

\begin{lemma}\label{cop} $\kj_n \subseteq \mn$ is completely order proximinal.
\end{lemma}
\begin{proof} By Proposition~\ref{proxchar}, we need to prove that for
  every $p \in \mathbb N,$ $\mpee(\mn)_+ + \mpee(\kj_n)$ is closed.
To this end assume that $P_k= (P_{ij}^k) \in \mpee(\mn)_+$ and $K_k=
(K_{ij}^k) \in \mpee(\kj_n)$ are sequences such that $\|H - P_k - K_k \|
\to 0,$ where $H= (H_{ij}).$ Applying the partial trace ${\rm id}_p \otimes
{\rm tr}$ leads to
 $\|( {\rm tr}(H_{ij}) - ({\rm tr}(P_{ij}^k)) \| \to 0$.
Consequently, the set $\{ ({\rm tr}(P_{ij}^k) \in \mpee: k \in \mathbb N \}$
is norm bounded. But this implies that the set of traces of this set
of matrices is bounded, which in turn implies that the set of positive
matrices $P_k$ has bounded trace and hence is a norm bounded
set. Thus, by dropping to a subsequence if necessary, we may assume
that $P_k$ converges in norm to some $P$ that is necessarily  in
$\mpee(\mn)_+$, since the latter set is closed. But then for this same
subsequence, we have that $K_k$ converges to an element $K \in
\mpee(\kj_n)$ and so $H = P +K \in \mpee(\mn)_+ + \mpee(\kj_n)$ and the
result follows.
\end{proof}

The following result gives a variety of characterisations of positivity for the matricial ordering of the operator system $\mn/\kj_n$.

\begin{proposition}\label{positive} The following statements are equivalent for matrices $A_{11},A_{ij}\in\mpee$, with $1\leq i,j\leq n$ and $j\not=i$:
\begin{enumerate}
\item\label{1} $\dot{1}\otimes A_{11}+\displaystyle\sum_{j\neq i}e_{ij}\otimes A_{ij}$ is positive in $(\mn/\kj_n)\otimes\mpee$;
\item\label{2} $\dot{1}\otimes A_{11}+\displaystyle\sum_{j\neq i}\dot{\psi}(e_{ij})\otimes A_{ij}$ is positive in $\mr\otimes\mpee$
for every $r\in\mathbb N$ and every ucp map
$\dot{\psi}:\mn/\kj_n\rightarrow\mr$;
\item\label{3} $1_r\otimes A_{11}+\displaystyle\sum_{j\neq i} \psi(E_{ij})\otimes A_{ij}$ is positive in $\mr\otimes\mpee$
for every $r\in\mathbb N$ and every ucp map
$\psi:\mn\rightarrow\mr$ such that $\psi(\kj_n)=\{0\}$;
\item\label{4} $1_r\otimes A_{11}+\displaystyle\sum_{j\neq i} \psi(E_{ij})\otimes A_{ij}$ is positive in $\mr\otimes\mpee$
for every $r\in\mathbb N$ and every ucp map
$\psi:\mn\rightarrow\mr$ such that $\psi(E_{ii})=\frac{1}{n}1_r$, for all $1\leq i\leq n$;
\item\label{5} $1_r\otimes A_{11}+\displaystyle\sum_{j\neq i} B_{ij}\otimes A_{ij}$ is positive in $\mr\otimes\mpee$
for every $r\in\mathbb N$ and every collection of matrices
$B_{ij}\in\mr$, $j\neq i$, with the property that
\[
\left[ \begin{array}{cccc} \frac{1}{n}1_r & B_{12} & \cdots & B_{1n} \\ B_{21} & \frac{1}{n}1_r & & \vdots \\ 
                                         \vdots & & \ddots & \vdots \\ B_{n1} & \dots & B_{n,n-1} & \frac{1}{n}1_r \end{array}\right]
 \]
is positive in  $\mr\otimes\mn$;    
\item\label{6} $1_r\otimes A_{11}+\displaystyle\sum_{j\neq i} \frac{1}{n}C_{ij}\otimes A_{ij}$ is positive in $\mr\otimes\mpee$
for every  $r\in\mathbb N$ and every collection of matrices
$C_{ij}\in\mr$, $j\neq i$, with the property that
\[
\left[ \begin{array}{cccc} 1_r & C_{12} & \cdots & C_{1n} \\ C_{21} & 1_r & & \vdots \\ 
                                         \vdots & & \ddots & \vdots \\ C_{n1} & \dots & C_{n,n-1} & 1_r \end{array}\right]
 \]
is positive in  $\mr\otimes\mn$;
\item\label{7} $1\otimes A_{11}+\displaystyle\sum_{j\neq i}(\frac{1}{n}w_iw_j^*)\otimes A_{ij}$ is positive in $\cstar(\fnl)\otimes\mpee$;
\item\label{8} $1\otimes (nA_{11})\,+\,\displaystyle\sum_{i=1} ^nE_{ii}\otimes B_i\,+\,\displaystyle\sum_{j\neq i}E_{ij}\otimes A_{ij}$ 
is positive in $\mn\otimes\mpee$ for some $B_1,\dots, B_n\in \mpee$ such that $\displaystyle\sum_{i=1}^nB_i\,=\,(n-n^2)A_{11}$.                                  
\item\label{9} $\sum_{i,j=1}^n E_{ij}\otimes R_{ij}$ is positive in $\mn\otimes\mpee$ for some $R_{ij} \in \mpee$ 
such that $R_{ij} = A_{ij}$ for $i \neq j$ and $R_{11} + \cdots + R_{nn} = n A_{11}.$
\end{enumerate}
\end{proposition}

\begin{proof} We shall prove that \eqref{1} $\Rightarrow$ \eqref{9} $\Rightarrow$ \eqref{8} $\Rightarrow$ ... $\Rightarrow$ \eqref{1}.

The hypothesis \eqref{1} is that $r=\dot{1}\otimes A_{11}+\displaystyle\sum_{j\neq i}e_{ij}\otimes A_{ij}$ is positive in $(\mn/\kj_n)\otimes\mpee$.
Because $\kj_n \subseteq M_n$ is completely order proximinal, there is a positive $R\in \mn\otimes\mpee$ such that $q\otimes{\rm id}_p(R)=r$. In writing
$R$ as $R=\sum_i E_{ii}\otimes R_{ii}+\sum_{j\neq i}E_{ij}\otimes R_{ij}$, we obtain from $r=\sum_{i,j}e_{ij}\otimes R_{ij}$ that $R_{ij}=A_{ij}$ for all $j\neq i$
and that $A_{11}=\frac{1}{n}\sum_iR_{ii}$. Thus \eqref{9} follows. To see \eqref{8}, set $R_{ii}=nA_{11}+B_i$, where $B_i=-\sum_{k\neq i}R_{kk}$, and so 
$nA_{11}=\sum_iR_{ii}=n^2A_{11}+\sum_iB_i$ implies that $\sum_iB_i=(n-n^2)A_{11}$, which is statement \eqref{8}.

Now assume \eqref{8}. Because
$1\otimes nA_{11}\,+\,\displaystyle\sum_{i=1}^n E_{ii}\otimes B_i\,+\,\displaystyle\sum_{j\neq i}E_{ij}\otimes A_{ij}$ 
is positive in $\mn\otimes\mpee$, the image of this matrix under the map $\phi\otimes{\rm id}_p$ is positive 
in $\cstar(\fnl)\otimes\mpee$, where $\phi:\mn\rightarrow\cstar(\fnl)$ is the ucp map \eqref{defn of phi} in Lemma \ref{norm}. 
That is, the following element is positive:
\[
\begin{array}{rcl}
&&1\otimes (nA_{11})+\displaystyle\sum_{i=1}^n \frac{1}{n}1\otimes B_i+ \displaystyle\sum_{j\neq i}\frac{1}{n}w_iw_j^*\otimes A_{ij} \\&& \\
&=& 1\otimes\left(nA_{11}+\frac{1}{n}\displaystyle\sum_{i=1}^n   B_i\right)+\displaystyle\sum_{j\neq i}\frac{1}{n}w_iw_j^*\otimes A_{ij} \\&& \\
&=& 1\otimes A_{11} + \displaystyle\sum_{j\neq i}\frac{1}{n}w_iw_j^*\otimes A_{ij}\,,
\end{array}
\]
because $\frac{1}{n}\displaystyle\sum_{i=1}^n   B_i=(1-n)A_{11}$. This establishes \eqref{7}.

 Next assume \eqref{7}: namely,
$1\otimes A_{11}+\displaystyle\sum_{j\neq i}(\frac{1}{n}w_iw_j^*)\otimes A_{ij}$ 
is positive in $\cstar(\fnl)\otimes\mpee$.
Let $r\in\mathbb N$ and suppose that the matrices $C_{ij}\in\mr$, $j\neq i$, satisfy
\begin{equation}\label{matrix C}
C\;=\;\left[ \begin{array}{cccc} 1_r & C_{12} & \cdots & C_{1n} \\ C_{21} & 1_r & & \vdots \\ 
                                         \vdots & & \ddots & \vdots \\ C_{n1} & \dots & C_{n,n-1} & 1_r \end{array}\right] \,\in \, (\mr\otimes\mn)_+\,.
 \end{equation}
Let $A=C^{1/2}$, the positive square root of $C$.  We may express $A$ in two ways:
\[
A\;=\;\left[ \begin{array}{c} a_1 \\ \hline a_2\\ \hline \vdots \\ \hline a_n\end{array} \right] \;=\; A^*\;=\;
\left[ \begin{array}{cccc} a_1^*& \vline\quad  a_2^* &\vline \quad \dots& \vline\quad  a_n^*\end{array} \right] \,,
\]
where each $a_j$ is a rectangular block matrix, or an $n$-tuple, of $r\times r$ matrices. Hence,
\[
C\;=\;AA^*\;=\;\left[ a_ia_j^*\right]_{1\leq i,j\leq n}\,.
\]
Because $a_ja_j^*=1$ for each $j$, the rectangular matrix $a_j^*$ is an isometry. Therefore, we may dilate $a_j^*$ to a unitary
which we denote by $u_j^*$
by way of the Halmos dilation:
\[
u_j^* \;=\; \left[
\begin{array}{cc}
a_j^* &(1-a_j^*a_j)^{1/2} \\
 0& -a_j
\end{array}
\right]  \in 
 \left[
\begin{array}{c|c}
{\mathcal M}_{nr,n}&  {\mathcal M}_{nr} \\
\cline{1-2}
 \mn & {\mathcal M}_{n,nr}
\end{array}
\right]
\;=\;{\mathcal M}_{(nr+n)}
\,.
\]
By the Universal Property, there is a homomorphism $\pi:\cstar(\fnl)\rightarrow{\mathcal M}_{(nr+n)}$
such that $\pi(w_j)=u_j^*$ for each $j$. If $j\neq i$, then
\[
\pi(w_iw_j^*)\;=\;      
\left[\begin{array}{cc} a_i^*a_j^* & * \\ *& * \end{array} \right]  \,.
\]
Therefore, if $\psi:\cstar(\fnl)\rightarrow{\rm M}_{(nr+n)}$ is the ucp map $\psi=v^*\pi v$, where
$v:\mathbb C^r\rightarrow \mathbb C^{nr+n}$ embeds $\mathbb C^r$ into the first $r$-coordinates of 
$\mathbb  C^{nr+n}$, then $\psi(w_iw^*_j)=C_{ij}$. Thus,
\[
\psi^{(p)}\left(1\otimes A_{11}+\displaystyle\sum_{j\neq i}(\frac{1}{n}w_iw_j^*)\otimes A_{ij}\right)
\;=\; 1_r\otimes A_{11}+\displaystyle\sum_{j\neq i}  \frac{1}{n}C_{ij}\otimes A_{ij}
\]
is positive in $\mr\otimes\mpee$, which yields statement \eqref{6}.

Assume \eqref{6} holds. For any set of matrices $B_{ij}\in \mr$, $j\neq i$, for which
\begin{equation}\label{matrix B}
B\;=\;
\left[ \begin{array}{cccc} \frac{1}{n}1_r & B_{12} & \cdots & B_{1n} \\ B_{21} & \frac{1}{n}1_r & & \vdots \\ 
                                         \vdots & & \ddots & \vdots \\ B_{n1} & \dots & B_{n,n-1} & \frac{1}{n}1_r \end{array}\right]
 \end{equation}
is positive in  $\mr\otimes\mn$, let $C_{ij}=n B_{ij}$. Thus, the matrix $B$ above is $\frac{1}{n}C$, where
$C$ is the matrix in \eqref{matrix C}. Therefore, from statement \eqref{6} one obtains \eqref{5}.

Assuming \eqref{5}, let $\psi:\mn\rightarrow\mr$ be a ucp map such that $\psi(E_{ii})=\frac{1}{n}1_r$ for all $1\leq i\leq n$.
Let $B_{ij}=\psi(E_{ij})$, $j\neq i$. Therefore, the Choi matrix of $\psi$, namely $[\psi(E_{ij})]\in \mr\otimes\mn$, is precisely
the matrix $B$ in equation \eqref{matrix B}. Because $\psi$ is completely positive, the Choi matrix $B$ of $\psi$
is positive 
in $\mr\otimes\mn$. Hence, \eqref{4} is implied by \eqref{5}.

Next assume \eqref{4}. If $\psi:\mn\rightarrow\mr$ is ucp and has the property that $\psi(\kj_n)=\{0\}$, then $E_{ii}-E_{jj}\in\kj_n$,
for all $j\neq i$. Because $\psi$ is unital, this implies that $\psi(E_{ii})=\frac{1}{n}1_r$, for each $i$, which yields \eqref{3}.

Assume \eqref{3} and let $\dot{\psi}:\mn/\kj_n\rightarrow\mr$ be a ucp map. The linear map $\psi:\mn\rightarrow\mr$ defined
by $\psi=\dot{\psi}\circ q$ is unital, annihilates $\kj_n$, and is the composition of cp maps $\dot{\psi}$ and $q$. 
Therefore, \eqref{2} follows directly from \eqref{3}.

Assume \eqref{2}: 
that is, $\dot{\psi}^{(p)}(h)$ is positive in $\mr\otimes\mpee$, for every $r\in\mathbb N$ and every ucp map
$\dot{\psi}:\mn/\kj_n\rightarrow\mr$, and where $h\in (\mn/\kj_n)\otimes \mpee$ is given by
$h=\dot{1}\otimes A_{11}+\displaystyle\sum_{j\neq i}e_{ij}\otimes A_{ij}$. In any operator system
$\oss$, an element $x\in\oss$ is positive if and only if $\omega(x)$ is positive in $\mr$, for every $r\in\mathbb N$ and
every ucp map $\omega:\oss\rightarrow\mr$ (see, for example, \cite[p.~178]{choi--effros1977}). Now let $\oss= (\mn/\kj_n)\otimes \mpee$
and choose an arbitrary ucp $\omega:\oss\rightarrow\mr$. Thus, there are ucp maps $\dot{\psi}_1,\dots,\dot{\psi}_m:\mn/\kj_n\rightarrow\mr$
and linear maps $\gamma_1,\dots,\gamma_m:\mathbb C^{p}\rightarrow\mathbb C^r\otimes\mathbb C^p$ such that
\[
\omega\,=\,\sum_{j=1}^m \gamma_j^*\dot{\psi}_j^{(p)}\gamma_j\,.
\]
Hence, \eqref{1} is implied by \eqref{2}.
\end{proof}

\begin{theorem}\label{mn mod jn} The ucp surjection $\phi:\mn\rightarrow\osw_n$  
is a complete quotient map and the C$^*$-envelope of $\mn/\kj_n$ is $\cstar(\fnl)$.
\end{theorem}

\begin{proof} By the First Isomorphism Theorem, there exists a ucp map $\dot{\phi}:\mn/\kj_n\rightarrow\cstar(\fnl)$
such that $\phi=\dot{\phi}\circ q$, where $\phi:\mn\rightarrow\cstar(\fnl)$ is the ucp map defined in equation \eqref{defn of phi}.
Therefore, $\phi(e_{ij})=\frac{1}{n}w_iw_j^*$, for all $1\leq i,j\leq n$, and so $\dot{\phi}$ maps $\mn/\kj_n$ surjectively onto
$\osw_n$. Hence, as $\mn/\kj_n$ and $\ost$ are vector
spaces of equal finite dimension, the surjection $\dot{\phi}$ is a linear isomorphism.
The equivalence of statements \eqref{1} and \eqref{7} in Proposition \ref{positive} shows, further, that
$\dot{\phi}$ is a complete order isomorphism of $\mn/\kj_n$ and $\osw_n$.

We now show that the C$^*$-envelope of $\mn/\kj_n$ is $\cstar(\fnl)$. 
Because there is a completely isometric 
embedding of $\mn/\kj_n$ into its C$^*$-envelope $\cstare(\mn/\kj_n)$ \cite[Chapter 15]{Paulsen-book},
we assume without loss generality that $\mn/\kj_n$ is an operator system in $\cstare(\mn/\kj_n)$ and that $\dot{1}$
is the multiplicative identity of $\cstare(\mn/\kj_n)$. We also 
suppose that $\cstar(\fnl)$ has been represented faithfully on a Hilbert space $\hil$. 
By Arveson's extension theorem, the ucp map $\dot{\phi}:\mn/\kj_n\rightarrow\cstar(\fnl)\subset\bh$
has a ucp extension $\varrho:\cstare(\mn/\kj_n)\rightarrow\bh$. Let $\varrho=v^*\pi v$ denote a minimal Stinespring
decomposition of $\varrho$. With respect to the decomposition of the representing Hilbert space $\hil_\pi$ as 
$\hil_\pi=\mbox{ran}\,v \oplus (\mbox{ran}\,v)^\bot$, every operator $\pi(ne_{ij})$ has the form
\[
\pi(ne_{ij})\;=\; \left[ \begin{array}{cc} w_iw_j^* & * \\ * &* \end{array}\right]\,.
\]
Because $\|ne_{ij}\|=1$ and $\pi$ is a contraction, the operator matrix above has norm exactly equal to $1$. And since
$w_iw_j^*$ is unitary, if either of the (1,2)- or (2,1)-entries of $\pi(ne_{ij})$ were nonzero, then the norm of $\pi(ne_{ij})$ would
exceed $1$. Thus,
\[
\pi(ne_{ij})\;=\; \left[ \begin{array}{cc} w_iw_j^* & 0 \\ 0 &* \end{array}\right]\,,
\]
which implies that $\varrho$ is multiplicative on the generators $e_{ij}$ of the C$^*$-algebra $\cstare(\mn/\kj_n)$. Hence,
$\varrho$ is a homomorphism that maps the generators of $\cstare(\mn/\kj_n)$ to the generators of $\cstar(\fnl)$, which implies that
$\varrho$ is surjective.

Because $\dot{\phi}$ is a unital complete order isomorphism, there is a unital isomorphism
$\alpha:\cstare(\mn/\kj_n)\rightarrow\cstare(\osw_n)$ such that $\alpha_{\vert \mn/\kj_n}=\pi_{\rm e}\circ\dot{\phi}$, where $\pi_{\rm e}$ is the 
unique surjective unital homomorphism $\cstar(\fnl)=\cstar(\osw_n)\rightarrow\cstare(\osw_n)$ arising from the Universal Property
of the C$^*$-envelope. Therefore, $\pi_e(w_j)$ is unitary, for each $2\leq j\leq n$, and hence so are $e_{1j}$ in $\cstare(\mn/\kj_n)$,
for $2\leq j\leq  n$. By the Universal Property of $\cstar(\fnl)$, there is a unital homomorphism $\beta:\cstar(\fnl)\rightarrow\cstare(\mn/\kj_n)$
such that $\beta(w_j)=e_{1j}$, $2\leq j\leq n$. Thus, $\beta\circ\varrho(w_j)=w_j$, for every $j$, which implies that $\varrho$ is in
fact an isomorphism.
\end{proof}

\begin{corollary} If $u_1,\dots, u_n$ are universal unitaries, then the 
C$^*$-envelope of the operator system $\mbox{\rm Span}\{u_iu_j^*\,:1\leq i,j\leq n\}$
is $\cstar(\fnl)$.
\end{corollary}
\begin{proof} Set $w_i = u_1^*u_i,$ and note that $w_2, \dots , w_n$ is a set of $n-1$ universal unitaries.
\end{proof}

Another application of Proposition \ref{positive} is the following theorem concerning the factorisation of
positive elements in the C$^*$-algebra $\cstar(\fn)\otimes\mpee$.

\begin{theorem} Let $u_1, \dots, u_n$ be universal unitaries and let $A_{ij} \in \mpee.$ 
Then $P=\sum_{i,j=1}^n u_iu_j^* \otimes A_{ij}$ is positive in $\cstar(\fn) \otimes \mpee$ if and 
only if there exist matrices $X_{ik} \in \mpee$, such that $P= \sum_{k=1}^n Y_kY_k^*$,
 where $Y_k = \sum_{i=1}^n u_i\otimes X_{ik}.$
\end{theorem}
\begin{proof} We have that the span of $u_iu_j^*$ is completely order isomorphic to $\mn/\kj_n$ via 
the map $u_iu_j^* \to e_{ij}$ with $e_{ii} = (1/n)\dot{1}.$
Thus, $P \to \dot{1}\otimes[1/n(A_{11} + \cdots + A_{nn})] + \sum_{i \neq j} e_{ij} \otimes A_{ij}.$ 
Apply Proposition~\ref{positive}\eqref{9} to lift this element to a positive element 
$R=\sum_{i,j=1}^n E_{ij} \otimes R_{ij}$ of $\mn\otimes\mpee.$  
Now factor $R= XX^*$ with $X= \sum_{i,j=1}^n E_{ij}\otimes X_{ij}$ and define $Y_k$ as above.

We have that 
\[ \sum_{k=1}^n Y_kY_k^* = \sum_{i,j,k=1}^n u_iu_j^* \otimes X_{ik}X_{jk}^*
= \sum_{i,j=1}^n u_iu_j^* \otimes R_{ij} = P\,,\]
and the result follows.
\end{proof}

To conclude this section, we determine below the dual operator system of $\osw_n$.

\begin{proposition}\label{e-wd} Consider the operator subsystem $\ose_n\subseteq\mn$ defined by
\begin{equation}\label{defn of ose}
\ose_n\,=\,\{A\in\mn\,:\,a_{ii}=a_{jj},\;1\leq i,j\leq n\}\,.
\end{equation}
Then the operator systems $\osw_n^d$ and $\ose_n$ are completely order isomorphic.
\end{proposition}

\begin{proof} The function $\phi:\mn\rightarrow\osw_n$ defined in \eqref{defn of phi} is a complete quotient map.
Therefore, the dual map $\phi^d:\osw_n^d\rightarrow\mn^d$ is a complete order injection (Proposition \ref{q-coi}).
The operator systems $\mn^d$ and $\mn$ are completely order isomorphic via the map
$S_{ij}\mapsto E_{ij}$, where $\{S_{ij}\}_{1\leq\i,j\leq n}$ is the basis of $\mn^d$ that
is dual to the basis $\{E_{ij}\}_{1\leq i,j\leq n}$ of $\mn$ \cite{paulsen--todorov--tomforde2010}. By Banach space duality,
$\phi^d(\osw_n^d)$ is the annihilator of $\kj_n$, namely $\ose_n$.
\end{proof}

\section{Tensor Products with $\mn/\kj_n$}\label{S:osw tensor products}

\begin{definition} Let $\oss$ be an operator system. 
\begin{enumerate}
\item For a fixed $n\in\mathbb N$,  $\oss$ is said to have \emph{property $\pstarn$} if, for every $p\in\mathbb N$ and all
$S_{11},S_{ij}\in \mpee(\oss)$, where $1\leq i,j\leq n$ and $j\neq i$, for which
\[
1\otimes S_{11}\,+\,\sum_{j\neq i}(\frac{1}{n}w_iw_j^*)\otimes S_{ij}\;\geq\;0 \mbox{ in }\; \cstar(\fnl)\omin\mpee(\oss)\,,
\]
then for every $\varepsilon>0$ there exist $R_{ij}^\varepsilon\in  \mpee(\oss)$, $1\leq i,j\leq n$,
such that
\begin{enumerate}
\item $R_\varepsilon=[R_{ij}^\varepsilon]_{1\leq i,j\leq n}$ is positive in $\mn(\mpee(\oss))$,
\item $R_{ij}^\varepsilon=S_{ij}$ for all $j\neq i$, and
\item $\displaystyle\sum_{i=1}^n R_{ii}^\varepsilon\,=\,n(S_{11}+\varepsilon 1_{\mpee(\oss)})$.
\end{enumerate}
\item We say that $\oss$ has \emph{property $\pstar$} if $\oss$ has property $\pstarn$ for every $n\in\mathbb N$.
\end{enumerate}
\end{definition}

Property $\pstar$ is suggested by Proposition \ref{positive} (equivalent
statements \eqref{7} and \eqref{9}), which shows that every matrix algebra $\mpee$ has property $\pstar$.
To say that an operator system $\oss$ has property $\pstarn$ is equivalent to saying that the map
$\phi\otimes{\rm id}_\oss:\mn\omin\oss\rightarrow\osw_n\omin\oss$ is
a complete quotient map (see Proposition \ref{prop psn} for a detailed argument).

The main results of this section are summarised by the following theorem.

\begin{theorem}\label{pstar thm} \mbox{\rm (Operator Systems with Property $\pstar$)}
\begin{enumerate}
\item\label{pstar-1} An operator system $\oss$ has property $\pstarn$ if and only if $(\mn/\kj_n)\omin\oss\,=\,(\mn/\kj_n)\omax\oss$.
\item\label{pstar-2} Every {\rm(min,max)}-nuclear operator system has property $\pstar$.
\item\label{pstar-3} $\bh$ has property $\pstar$.
\item\label{pstar-4} If an operator system $\oss$ has property $\pstarn$, 
then $\cstar(\fnl)\omin\oss\,=\,\cstar(\fnl)\omax\oss$.
\end{enumerate}
\end{theorem}

\begin{corollary}\label{k thm}{\rm (Kirchberg's Theorem)} 
$\cstar(\fni)\omin\bh\,=\,\cstar(\fni)\omax\bh$.
\end{corollary}

\begin{proof} Recall that the free groups embed into one another as subgroups; in particular, if $n\geq2$, then
$\fnl\subset\fn$ and
$\fni\subset\fn$ as subgroups. These subgroup emebeddings lead to C$^*$-algebra
embeddings $\cstar(\fnl)\subset\cstar(\fn)$
and $\cstar(\fni)\subset\cstar(\fn)$, for $n\geq2$.
Hence, every
representation of the algebraic tensor product $\cstar(\fn)\otimes\bh$ induces (by restriction)
a representation of
$\cstar(\fni)\otimes\bh$, and conversely every representation of $\cstar(\fni)\otimes\bh$
induces a representation of $\cstar(\fn)\otimes\bh$. Thus, in consideration of the definition of the max-norm, 
if $a\in\cstar(\fn)\otimes\bh$, then the max-norm of $a$ in
$\cstar(\fni)\omax\bh$ coincides with the max-norm of $a$
in $\cstar(\fn)\omax\bh$.
That is, 
$\cstar(\fn)\omax\bh\subset\cstar(\fni)\omax\bh$ as a C$^*$-subalgebra.
As $\cstar(\fni)\omax\bh\subseteq\cstar(\fni)\omin\bh$, we show below that there is a
monomorphism of $\rho:\cstar(\fni)\omin\bh\rightarrow\cstar(\fni)\omax\bh$. 
 
By \eqref{pstar-3} and \eqref{pstar-4} of Theorem \ref{pstar thm},  
$\cstar(\fn)\omin\bh=\cstar(\fn)\omax\bh\subset\cstar(\fni)\omax\bh$ for every $n\in\mathbb N$, and so for each $n$
there is a monomorphism $\rho_n:\cstar(\fn)\omin\bh\rightarrow\cstar(\fni)\omax\bh$ such that $\rho_{n-1}=\rho_n\circ\iota_{n-1}$,
where $\iota_{n-1}:\cstar(\fnl)\omin\bh\rightarrow\cstar(\fn)\omin\bh$ is the canonical inclusion. Because $\cstar(\fni)\omin\bh$ is
the direct limit of the directed system $(\cstar(\fn)\omin\bh,\iota_n)_n$, there is a unique monomorphism 
$\rho:\cstar(\fni)\omin\bh\rightarrow\cstar(\fni)\omax\bh$ such that $\rho_n=\rho\circ\mathfrak i_n$, where $\mathfrak i_n$
is the embedding of $\cstar(\fn)\omin\bh$ into $\cstar(\fni)\omin\bh$ that is compatible with the inclusions $\iota_n$.
Hence, $\cstar(\fni)\omin\bh=\cstar(\fni)\omax\bh$.
\end{proof}

The following proposition yields assertion \eqref{pstar-1} of Theorem \ref{pstar thm}.

\begin{proposition}\label{pstar-1 pf}  $(\mn/\kj_n)\omin\oss\,=\,(\mn/\kj_n)\omax\oss$ if and only if
$\oss$ has property $\pstarn$.
\end{proposition}

\begin{proof} Observe that if $p\in\mathbb N$ and $H\in \mpee(\mn/\kj_n\otimes\oss)$, then $H=[h_{k\ell}]_{1\leq k,\ell\leq p}$ and
\[
h_{k\ell}\;=\;\dot{1}\otimes S_{11}^{(k,\ell)}\,+\,\sum_{\ell\neq k}e_{k \ell}\otimes S_{ij}^{(k,\ell)}\,
\]
for some $S_{11}^{(k,\ell)},S_{ij}^{(k,\ell)}\in\oss$, where $1\leq k,\ell\leq p$ and $j\neq i$. Therefore, we write $H$ as
\[
H\;=\;\dot{1}\otimes S_{11}\,+\, \sum_{\ell\neq k}e_{ij}\otimes S_{ij}\,
\]
where
\[
S_{11}\;=\;[S_{11}^{(k,\ell)}]_{1\leq k,\ell\leq p}\quad\mbox{and}\quad
S_{ij}\;=\;[S_{ij}^{(k,\ell)}]_{1\leq k,\ell\leq p}\,.
\]
Hence, to prove that $\oss$ has property $\pstarn$, any argument that shows the property for $p=1$
also shows the property for arbitrary $p\in\mathbb N$.

Assume that $(\mn/\kj_n)\omin\oss\,=\,(\mn/\kj_n)\omax\oss$. We may also assume $p=1$ by the observation above.
Suppose that 
\[
{1}\otimes S_{11}\,+\,\sum_{j\neq i}(\frac{1}{n}w_iw_j^*)\otimes S_{ij}\,\in\,\left(  \cstar(\fnl)\omin\oss \right)_+\,.
\]
The ucp map $\phi:\mn\rightarrow\osw_n$ given by \eqref{defn of phi} is a complete quotient map, and therefore the map
\[
\dot{\phi}\otimes{\rm id}_{\oss}: (\mn/\kj_n)\omax\oss \rightarrow \osw_n\omax\oss
\]
is a complete order isomorphism. Hence, by the assumption that $(\mn/\kj_n)\omin\oss=(\mn/\kj_n)\omax\oss$,
we deduce that
\[
h\;=\;\dot{1}\otimes S_{11}\,+\,\sum_{j\neq i}e_{ij}\otimes S_{ij}
\]
is positive in $(\mn/\kj_n)\omin\oss$. By hypothesis,
$h\in ((\mn/\kj_n)\omax\oss)_+$ and so,
for every $\varepsilon>0$,
\[
h+\varepsilon 1\,=\,\left( \dot{1}\otimes(S_{11}+\varepsilon 1)+\sum_{j\neq i}e_{ij}\otimes S_{ij}\right)
\in \dmaxone(\mn/\kj_n,\oss)\,.
\]
Thus, there are $A\in\mk(\mn/\kj_n)_+$, $C\in\mm(\oss)_+$, and 
$\alpha:\mathbb C^k\otimes\mathbb C^m\rightarrow\mathbb C$ linear such that
\[
h+\varepsilon 1\,=\,\alpha(A\otimes C)\alpha^*\,.
\]
By Proposition \ref{positive}, the fact that $A\in\mk(\mn/\kj_n)_+=(\mn/\kj_n\otimes\mk)_+$ implies that there exist
$R=[R_{ij}]\in(\mn\otimes\mk)_+$ such that $q^{(k)}(R)=A$. Thus,
\[
h+\varepsilon 1\;=\;\alpha(q(R)\otimes C)\alpha^*\;=\;q\otimes{\rm id}_\oss\left(\alpha(R\otimes C)\alpha^*\right)
\in \dmaxone(\mn,\oss)\,.
\]
That is, $h+\varepsilon 1=q\otimes{\rm id}_\oss(R_\varepsilon)$, where 
$R_\varepsilon=\alpha(R\otimes C)\alpha^*\in(\mn\omax\oss)_+=\mn(\oss)_+$.
Therefore, we have proved that for every $\varepsilon>0$ there exists $R_\varepsilon\in \mn(\oss)_+$ such that
\[
\dot{1}\otimes(S_{11}+\varepsilon 1)+\sum_{j\neq i}e_{ij}\otimes S_{ij}\;=\;\sum_ie_{ii}\otimes R_{ii}^\varepsilon+\sum_{j\neq i}e_{ij}\otimes R_{ij}^\varepsilon\,,
\]
which yields $R_{ij}^\varepsilon=S_{ij}$ for $j\neq i$ and $\frac{1}{n}\sum_iR_{ii}^\varepsilon=S_{11}+\varepsilon 1$. Hence, 
$\oss$ has property $\pstarn$.

Conversely, suppose that $\oss$ has property $\pstarn$.
To prove that $\cminp(\mn/\kj_n,\oss)\subseteq\cmaxp(\mn/\kj_n,\oss)$ for all $p\in\mathbb N$,
we may, as mentioned earlier, restrict ourselves to the case $p=1$.

Suppose that $h\in \left(\mn/\kj_n\omin\oss\right)_+$. Thus, 
\[
h\;=\;\dot{1}\otimes S_{11}\,+\,\sum_{j\neq i}e_{ij}\otimes S_{ij}
\]
for some $S_{11},S_{ij}\in \oss$, where $1\leq i,j\leq n$ and $j\neq i$.  Let $\phi:\mn\rightarrow\cstar(\fnl)$
be the ucp map defined by \eqref{defn of phi}. Thus, 
\[
\dot{\phi}\otimes {\rm id}_{\oss}:\mn/\kj_n\omin\oss \rightarrow \cstar(\fnl)\omin\oss
\]
sends $h$ to 
\[
{1}\otimes S_{11}\,+\,\sum_{j\neq i}(\frac{1}{n}w_iw_j^*)\otimes S_{ij}\,\in\,\left(  \cstar(\fnl)\omin\oss \right)_+\,.
\]
By hypothesis, $\oss$ has property $\pstar$, and so 
for every $\varepsilon>0$ there exist $R_{ij}^\varepsilon\in  \oss$, $1\leq i,j\leq n$,
such that
$R_\varepsilon=[R_{ij}^\varepsilon]_{1\leq i,j\leq n}\in \mn(\oss)_+$, $R_{ij}^\varepsilon=S_{ij}$ for all $j\neq i$,  and
$\displaystyle\sum_{i=1}^n R_{ii}^\varepsilon\,=\,n(S_{11}+\varepsilon 1_{\oss})$.
The cone $\mn(\oss)_+$ is $(\mn\omax\oss)_+$ and therefore the ucp map 
\[
q\otimes{\rm id}_{\oss}: \mn\omax\oss\rightarrow (\mn/\kj_n)\omax\oss
\]
sends $R_\varepsilon$ to the positive element
\[
\sum_{i,j=1}^n e_{ij}\otimes R_{ij}^\varepsilon\;=\; \varepsilon(\dot{1}\otimes 1_{\oss})\,+\,h \,\in\, \left((\mn/\kj_n)\omax\oss\right)_+\,.
\]
As this holds for all $\varepsilon>0$, this implies,
by the Archimedean property of the positive cone in operator systems, that
$h\in  \left((\mn/\kj_n)\omax\oss\right)_+$.
\end{proof}
  
By definition, an operator system $\oss$ is (min,max)-nuclear if
$\oss\omin\ost=\oss\omax\ost$ for every operator system $\ost$. Thus,
statement \eqref{pstar-2} of Theorem \ref{pstar thm} follows immediately:
 
\begin{corollary}\label{pstar-3 pf} If $\oss$ is {\rm (min,max)}-nuclear, then $\oss$ has property $\pstar$.
\end{corollary} 

We now prove statement \eqref{pstar-3} of Theorem \ref{pstar thm}.

\begin{proposition}\label{pstar-3 pf} $\bh$ has property $\pstar$.
\end{proposition}

\begin{proof} 
Suppose that $S_{11}, S_{ij}\in \bh$, $1\leq i, j\leq n$ and $j\neq i$, and that
\[
{1}\otimes S_{11}\,+\,\sum_{j\neq i}(\frac{1}{n}w_iw_j^*)\otimes S_{ij}\,\in\,\left(  \cstar(\fnl)\omin\bh \right)_+\,.
\]
(We are assuming, as in the proof of Proposition \ref{pstar-1 pf}, that $p=1$.) Let $\varepsilon>0$.

Because $\bk$ is nuclear for all finite-dimensional Hilbert spaces $\hilk$, 
if $p_\ssL\in \bh$ is a projection onto a  finite-dimensional subspace $\ssL\subseteq\hil$, then 
there exist $R_{ij}^{\varepsilon,\ssL}\in  \bh$, $1\leq i,j\leq n$, such that
$R_{\varepsilon,\ssL}=[R_{ij}^{\varepsilon,\ssL}]_{1\leq i,j\leq n}\in \mn(\bh)_+$,
$R_{ij}^{\varepsilon,\ssL}=p_\ssL S_{ij} p_\ssL$ for all $j\neq i$,  and
$\displaystyle\sum_{i=1}^n R_{ii}^{\varepsilon,\ssL}\,=\,n(p_\ssL S_{11}p_\ssL+\varepsilon p_{\ssL})$.
Therefore, $\mathcal F = \{R_{\varepsilon,\ssL}\}_{\ssL}$ is a net in $\mn(\bh)_+$ under subspace inclusion.
Moreover, the fact that 
$\displaystyle\sum_{i=1}^n R_{ii}^{\varepsilon,\ssL}\,=\,n(p_\ssL S_{11}p_\ssL+\varepsilon p_{\ssL}) \leq n(S_{11}+\varepsilon 1)$
implies that the diagonal operators of each matrix $R_{\varepsilon,\ssL}$ are bounded. 
By the positivity of 
$R_{\varepsilon,\ssL}$, the off-diagonal operators are also bounded, and hence $\mathcal F$ is a bounded net
of positive operators. Let $R_\varepsilon=[R_{ij}^\varepsilon]\in\mn(\bh)_+$ be the limit of some weakly convergent subnet 
$\mathcal F'= \{R_{\varepsilon,\ssL'}\}_{\ssL'}$ of $\mathcal F$. Then $p_{\ssL'}\rightarrow 1$ strongly and, therefore,
$\displaystyle\sum_{i=1}^n R_{ii}^\varepsilon\,=\,n(S_{11}+\varepsilon 1)$.
Because $R_{ij}^\varepsilon=S_{ij}$ for all $j\neq i$, the proof is complete.
\end{proof}

The proof below of statement \eqref{pstar-4} of Theorem \ref{pstar thm} is the final step in the proof of this theorem.

\begin{proposition}\label{pstar-4 pf}  If $\oss$ has property $\pstarn$, 
then $\cstar(\fnl)\omin\oss\,=\,\cstar(\fnl)\omax\oss$.
\end{proposition}

\begin{proof} By \cite[Theorem 6.7]{kavruk--paulsen--todorov--tomforde2010},
$\cstar(\fnl)\oc\oss\,=\,\cstar(\fnl)\omax\oss$; therefore, we aim to show that 
if $x\in\mn((\cstar(\fnl)\omin\oss))_+$, then $\phi\cdot\psi^{(n)}(x)\in\mn(\bh)_+$ for every
pair of completely positive maps $\phi:\cstar(\fnl)\rightarrow\bh$ and $\psi:\oss\rightarrow\bh$
with commuting ranges.  
Let $\phi=v^*\pi v$ be a minimal Stinespring representation of $\phi$, where 
$\pi:\cstar(\fnl)\rightarrow\bhpi$ is a unital representation. By the Commutant Lifting Theorem
\cite[Theorem 1.3.1]{arveson1969}, there is a unital homomorphism 
$\delta:\phi(\cstar(\fnl))'\rightarrow\bhpi$ such that $\delta(z)v=vz$. for all $z\in\phi (\cstar(\fnl))'$.
Because $\psi(\oss)$ lies within the commutant of $\phi(\cstar(\fnl))$,
we obtain a completely positive map $\tilde\psi=\delta\circ\psi:\oss\rightarrow\bhpi$
such that the range of $\tilde\psi$ commutes with the range of $\pi$.

By hypothesis,
$\oss$ has property $\pstarn$ and so, by Proposition \ref{pstar-1 pf} and by the fact that $\mn/\kj_n$ and $\osw_n$
are completely order isomorphic, $\osw_n\omin\oss=\osw_n\omax\oss$. 
Consider the ucp map $\pi\cdot\tilde\psi:\osw_n\omin\oss\rightarrow\bhpi$. Because
$\osw_n\omin\oss$ is ucp mapped into $\cstar(\fnl)\omin\cstare(\oss)$ and because $\bhpi$ is
injective, there is a ucp extension $\gamma:\cstar(\fnl)\omin\cstare(\oss)\rightarrow\bhpi$ of $\pi\cdot\tilde\psi$.
For each $i$, $\gamma(w_i\otimes 1)=\pi\cdot\tilde\psi(w_i\otimes 1)=\pi(w_i)$,
and so $w_i\otimes 1$ is in the multiplicative domain $\mathfrak M_\gamma$ of $\gamma$. Hence,
$\cstar(\fnl)\otimes\{1\}=\cstar(\osw_n)\otimes\{1\}\subseteq\mathfrak M_\gamma$ and, therefore,
\[
\gamma(a\otimes s)=\gamma((a\otimes1)(1\otimes s))=\gamma(a\otimes 1)\gamma(1\otimes s)=\pi(a)\tilde\psi(s)\,,
\]
for all $a\in\cstar(\fnl)$, $s\in\oss$. Thus, on elementary tensors,
\[
\phi\cdot\psi(a\otimes s)=\phi(a)\psi(s)=v^*\pi(s)v\psi(s)=v^*\pi(s)\tilde\psi(s)v=v^*\gamma(a\otimes s)v\,.
\]
Therefore,
$\phi\cdot\psi={\rm Ad}_v\circ\gamma_{\vert\cstar(\fnl)\omin\oss}$ and so $(\phi\cdot\psi)^{(n)}(x)\in \mn(\bh)_+$.
 \end{proof}

\section{The Operator System $\tn/\kj_n$}\label{S:oss}
 
\begin{definition} Let $\tn=\{A\in\mn\,:\,A_{ij}=0,\;\forall\,|i-j|\geq2\}$, the operator system
of \emph{tridiagonal matrices}.
\end{definition}

Observe that $\tn=\mbox{Span}\,\{E_{ij}\,:\,|i-j|\leq 1\}$ and that
$\tn$ contains the kernel $\kj_n=\{D\in\mn\,:\,D\;\mbox{is diagonal and } {\rm tr} \,D=0\}$ of the
ucp map $\phi:\tn\rightarrow\cstar(\fnl)$ defined in \eqref{defn of phi} by $\phi(E_{ij})=\frac{1}{n}w_iw_j^*$.
Our interest in this section is with the quotient operator system $\tn/\kj_n$.

We begin with a useful fact about the operator system $\tn$ itself.

\begin{proposition}\label{tn is (min,c)-nucl} $\tn\omin\oss=\tn\oc\oss$ for every operator system $\oss$. That is, $\tn$ is {\rm (min,c)}-nuclear.
\end{proposition}

\begin{proof}
Let $G=(V,E)$ be the graph with vertex set $V=\{1,\dots, n\}$ and edge set $E=\{(i,j)\,:\,|i-j|=1\}$. Thus, 
$G$ is simply a line segment from vertex $1$ through to vertex $n$. The operator system $\oss_G\subset\mn$ of the
graph $G$ is the span of matrix units $E_{ij}\in\mn$ for which $(i,j)\in E$, and so $\oss_G=\tn$. Because $G$ is
a chordal graph, $\tn=\oss_G$ is (min,c)-nuclear \cite[Proposition 6.10]{kavruk--paulsen--todorov--tomforde2009}.
\end{proof}

\begin{theorem}\label{tn 1}
Assume that $u_1,\dots,u_{n-1}$ are $n-1$ universal unitaries 
that generate the  full group C$^*$-algebra $\cstar(\fnl)$. Let $u_0=1$, $u_{-j}=u_j^*$, and
\[
\oss_{n-1}\,=\,\mbox{\rm Span}\,\{1,u_j, u_j^*\,:\,1\leq j\leq n-1\}\;=\;\mbox{\rm Span}\,\{u_j\,:\,-n+1\leq j\leq n-1\}\,.
\]
Consider the function $\phi:\tn\rightarrow\oss_{n-1}$ defined by
\begin{equation}\label{defn of phi 2}
\phi(E_{ij})\;=\;\frac{1}{n}u_{j-i},\,\;\forall\,|i-j|\leq 1\,.
\end{equation}
Then:
\begin{enumerate}
\item $\ker\phi=\kj_n$, the vector subspace of all diagonal matrices 
$D\in\mn$ with $\mbox{\rm tr}(D)=0$, 
\item $\kj_n\subset \tn$ is completely order proximinal, and
\item $\phi$ is a complete quotient map.
\end{enumerate}
\end{theorem}

\begin{proof} Because $\kj_n$ is a kernel in $\mn$ and $\kj_n\subset\tn\subset\mn$, 
Proposition \ref{quo lemma} shows that we may assume without loss of generality that $\tn/\kj_n\subset\mn/\kj_n$.
The function $\phi$ defined in \eqref{defn of phi 2} above is simply the function $\phi_{\vert \tn}$, the restriction of the the function
$\phi$ defined previously in \eqref{defn of phi} to the operator subsystem $\tn$, but where we identify each product 
$w_iw_{i+1}^*$ with $u_i$, for all $1\leq i\leq(n-1)$. (Assuming that $w_2,\dots, w_n$ $n-1$ are universal unitaries, then so are 
$u_1,\dots, u_{n-1}$.) Therefore, the restriction of $\phi$ from $\mn$ to $\tn$ preserves complete positivity. The proof of Lemma \ref{cop}
shows that $\kj_n\subset \tn$ is completely order proximinal, and it is evident that the restriction
of $\dot{\phi}$ from $\mn/\kj_n$ to $\tn/\kj_n$ is a complete order isomorphism of $\tn/\kj_n$ and $\oss_{n-1}$.
\end{proof}

\begin{corollary}\label{pos s} The following statements are equivalent for $A_i\in\mpee$, $-n+1\leq i\leq n-1$:
\begin{enumerate}
\item\label{pos s 1} $\displaystyle\sum_{i=1-n}^{n-1} u_i\otimes A_i\,\in\,\mpee(\oss_{n-1})_+$;
\item\label{pos s 2} there exists $R=[R_{ij}]\in \mpee(\mn)_+$ such that $R_{ij}=0$ if $|i-j|\geq 2$, 
$R_{i,i+1}=A_i$ and $R_{i+1,i}=A_{-i}$, for all $1\leq i\leq (n-1)$, and $\displaystyle\sum_{i=1}^n R_{ii}=A_0$.
\end{enumerate}
\end{corollary}

\begin{proof} By statements \eqref{7} and \eqref{9} of Proposition \ref{positive}, the element
\[
\displaystyle\sum_{i=1-n}^{n-1} u_i\otimes A_i\,=\,1\otimes A_0+\sum_{i=1}^{n-1}(\frac{1}{n}w_{i+1}w_{i}^*)\otimes (n A_{-i})
+\sum_{i=1}^{n-1}(\frac{1}{n}w_{i}w_{i+1}^*)\otimes (n A_{i})
\]
is positive in $\cstar(\fnl)\otimes\mpee$ if and only if there is a $\tilde R=[\tilde R_{ij}]_{i,j}\in \mn(\mpee)_+$
with $\tilde R_{ij}=0$ if $|i-j|\geq 2$, 
$\tilde R_{i,i+1}=nA_i$ and $\tilde R_{i+1,i}=nA_{-i}$, for all $1\leq i\leq (n-1)$, and $\displaystyle\sum_{i=1}^n \tilde R_{ii}=nA_0$. 
This last statement is equivalent to the matrix $R=\frac{1}{n}\tilde R$ having the properties in \eqref{pos s 2} above.
\end{proof}

\begin{theorem}\label{dual of tn} Let 
$\osu_n$ and $\osv_n$ be the operator subsystems of
$\displaystyle\bigoplus_{k=1}^{n-1}\mtwo$ defined by
\begin{equation}\label{defn of osv}
\osv_n\,=\,\left\{\bigoplus_{k=1}^{n-1}\left[ \begin{array}{cc} a_{11}^{k} & a_{12}^{k} \\ a_{21}^{k} & a_{22}^{k}\end{array}\right]\,:\,
a_{22}^{k}=a_{11}^{k+1},\;\forall\,1\leq k\leq (n-1)
\right\}
\end{equation}
and
\begin{equation}\label{defn of osu}
\osu_n\,=\,\left\{\bigoplus_{k=1}^{n-1}\left[ \begin{array}{cc} a_{11}^{k} & a_{12}^{k} \\ a_{21}^{k} & a_{22}^{k}\end{array}\right]\,:\,
a_{ii}^{k}=a_{jj}^{\ell},\;\forall\,1\leq k,\ell\leq (n-1)\,,i,j=1,2
\right\}\,.
\end{equation}
Then
\begin{enumerate}
\item $\tn^d$ and $\osv_n$ are completely order isomorphic, and
\item $\oss_{n-1}^d$ and $\osu_n$ are completely order isomorphic.
\end{enumerate}
\end{theorem}

\begin{proof} The canonical linear basis of $\tn$ is given by the subset $\{E_{ij}\}_{|i-j|\leq 1}$ of the matrix units
of $\mn$. Let $\{S_{ij}\}_{|i-j|\leq1}\subset\tn^d$
be the basis of $\tn^d$ that is dual to this canonical basis of $\tn$.

Consider the operator system $\displaystyle\bigoplus_{k=1}^{n-1}\mtwo$ and 
for each $1\leq k\leq (n-1)$ define a ucp map  $\rho_k: \displaystyle\bigoplus_{k=1}^{n-1}\mtwo\rightarrow\mtwo$ by projection 
onto the $k$-th coordinate. Let $\{f_j\}_{j=1}^{2n-2}$ and $\{e_1,e_2\}$ be the canonical orthonormal bases of $\mathbb C^{2n-2}$
and $\mathbb C^2$, respectively, and let $\gamma_k:\mathbb C^{2n-2}\rightarrow\mathbb C^2$ be the linear map defined by
$\gamma_k(f_{2k})=e_1$, $\gamma_k(f_{2k+1})=e_2$, and $\gamma_k(f_j)=0$ for all other $j$. Let
$\theta:\displaystyle\bigoplus_{k=1}^{n-1}\mtwo\rightarrow\tn$ be the ucp map defined by
$\theta=\sum_{k}\gamma_k^*\rho_k\gamma_k$. Thus,
\[
\theta\left((A_k)_k\right)\,=\,\displaystyle\sum_{k=1}^{n-1}\displaystyle\sum_{i,j=1}^2 a_{ij}^kE_{i+k-1, j+k-1}\,, 
\]
where $A_k=[a_{ij}^k]_{i,j}\in\mtwo$. The dual map $\theta^d:\tn^d\rightarrow\displaystyle\bigoplus_{k=1}^{n-1}\mtwo$ is given by
\[
\theta^d\left(\displaystyle\sum_{|i-j|\leq 1}a_{ij}S_{ij}\right)\;=\;
\bigoplus_{k=1}^{n-1}\left[ \begin{array}{cc} a_{kk} & a_{k\,k+1} \\ a_{k+1\,k} & a_{k+1\,k+1}\end{array}\right]\,.
\]
Now let $\psi=\theta^d$ and observe that the range of $\psi$ is precisely $\osv_n$.
Thus, $\psi$ is a completely positive linear isomorphism between $\tn^d$ and $\osv_n$. We now show that
$\psi^{-1}$ is completely positive. 

To this end, suppose that $X\in \mpee(\tn^d)$ and $Y=\psi^{(p)}(X)\in \mpee(\osv_n)_+$. Our aim is to prove that 
$X$ is positive, which is to say that $X$ is a positive linear functional on $\mpee(\tn)$. The matrix $Y$ 
is a $p\times p$
matrix of $n\times n$ tridiagonal matrices; hence,
$y$ is an $n\times n$ tridiagonal matrix of $p\times p$ matrices $y_{ij}^k\in\mpee$:
\[
Y\,=\,\bigoplus_{k=1}^{n-1}\left[ \begin{array}{cc} y_{kk} & y_{k\,k+1} \\ y_{k\,k+1}^* & y_{k+1\,k+1}\end{array}\right]\,.
\]
The pull back of $Y$ to $X\in \mpee(\tn^d)$ is 
\[
X\,=\,\left[ \begin{array}{ccccc} y_{11} & y_{12} & && \\ 
                                                  y_{12}^* & y_{22} & y_{23} & & \\
                                                  & y_{23}^* & y_{33} & y_{34} & \\
                                                  &&&\ddots& \\
                                                  &&&&y_{nn} \end{array}
\right]\,.
\]
View the matrix $X$ above as a partial matrix in the sense that outside the tridiagonal band the
entries are not specified. The only fully specified square submatrices of $X$ are the ones corresponding to
$2\times 2$ principal submatrices $\left[ \begin{array}{cc} y_{kk} & y_{k\,k+1} \\ y_{k\,k+1}^* & y_{k+1\,k+1}\end{array}\right]$,
each of which is a direct summand of the positive matrix $Y$. Hence, the partially specified matrix $X$ can be completed 
to a positive matrix $\tilde X$ (see, for example, \cite[Theorem 4.3]{paulsen--power--smith1989}). The action of $\tilde X$
on $\mn(\mpee)$ as a linear functional is given by $Z\mapsto\mbox{tr}(Z\tilde X)$, and so $X$ is a restriction
of  $\tilde X$ as a linear functional on $\tn$. Therefore, $X$ is also positive, 
which completes the proof that $\psi$ is a complete order isomorphism.

By Thorem \ref{tn 1}, $\phi:\tn\rightarrow\oss_{n-1}$ is a complete quotient map; therefore, $\psi\circ\phi^d:\oss_{n-1}^d\rightarrow\osv_n$
is a complete order injection. We need only identify the range of $\psi\circ\phi^d$. Since $\oss_{n-1}$ is completely order isomorphic 
to $\tn/\kj_n$, the dual of $\oss_{n-1}$ is the annihilator of $\kj_n$ in $\mn^d$. Hence, $\psi\circ\phi^d(\oss_{n-1})=\osu_n$.
\end{proof}

\section{The Kirchberg Problem and the Connes Embedding Problem}\label{S:kp-cep}

Recall that the operator systems $\ose_n$, $\osv_n$, and $\osu_n$ of $n\times n$ matrices
are defined in \eqref{defn of ose}, \eqref{defn of osv}, and \eqref{defn of osu} respectively.
In this section we show that some matrix theory problems involving these operator systems,
if solved affirmatively, would lead to an affirmative solution to Kirchberg's Problem and, hence, to
the Connes Embedding Problem as well. 

The following lemma is a proof technique that we shall require.

\begin{lemma}\label{diagram lemma} Suppose that $\osr$, $\oss$, $\ost$, and $\osu$ are operator systems and that,
for linear transformations $\psi$, $\theta$, $\mu$, and $\nu$, 
where $\nu$ is a complete quotient map, $\mu$ is a complete order isomorphism, $\theta$ is a linear isomorphism, 
and $\theta^{-1}$
is completely positive,
the following diagram is commutative:
\begin{equation}\label{e:diagram}
\begin{CD}
\osr @>{\mu}>>\osu \\
@V{\psi}VV           @VV{\nu}V  \\
\oss   @>>{\theta}> \ost\,.
\end{CD}
\end{equation}
Then $\psi$ is a complete quotient map if and only if $\theta$ is a complete order isomorphism.
\end{lemma}

\begin{proof} Suppose that $\psi$ is a complete quotient map. Thus, if $y\in\mpee(\oss)_+$, then
there is a hermitian $x\in\mpee(\osr)$ such that for every $\varepsilon>0$ there is a $k_\varepsilon\in\ker\psi$ such that 
$x'=\varepsilon1_\osr+x+k_\varepsilon\in\mpee(\osr)_+$. We have $\ker\psi\subseteq\ker(\nu\circ\mu)$, as $\theta\circ\psi=\nu\circ\mu$;
thus, $\varepsilon\nu\circ\mu(1)+\nu\circ\mu(x)=\nu\circ\mu(x')\in\mpee(T)_+$. 
As this is true for every $\varepsilon>0$, $\nu\circ\mu(x)=\theta\circ\psi(x)=\theta(s)\in\mpee(T)_+$, which proves that
$\theta$ is cp and, hence, that $\theta$ is a complete order isomorphism.

Conversely, if $\theta$ is a complete order isomorphism, then $\psi=\theta^{-1}\circ\nu\circ\mu$. Because
$\nu$ is a complete quotient map and $\theta^{-1}$ and $\mu$ are complete order isomorphisms, one easily deduces
that $\psi$ is a complete quotient map.
\end{proof}

We show in Theorems \ref{w-kp}\eqref{w-kp-3} and \ref{s1-kp}\eqref{s1-kp-2} below that an affirmative solution of 
two problems concerning the tensor products of operator systems of matrices will result in an
affirmative solution to the Kirchberg Problem.

Consider the following commutative diagram of vector spaces and linear transformations:
\begin{equation}
\begin{CD}
\mn\omin\mn   @>{ \cong}>>\mn\omax\mn \\
@V{\phi\otimes\phi}VV           @VV{\phi\otimes\phi\quad\mbox{\rm (complete quotient map)}}V  \\
\osw_n\omin\osw_n   @>>{{\rm id}_{\osw_{n}\otimes\osw_{n}}}> \osw_n\omax\osw_n\,.
\end{CD}
\end{equation}

Above, the identity map on $\mn\otimes\mn$ is a complete order isomorphism between $\mn\omin\mn$ and $\mn\omax\mn$,
and the map $\phi\otimes\phi:\mn\omax\mn\rightarrow\osw_n\omax\osw_n$ is a complete quotient map by
Corollary \ref{max quo cor}. The map $\theta={\rm id}_{\osw_{n}\otimes\osw_{n}}$ is a linear isomorphism 
of $\osw_n\omin\osw_n$ and $\osw_n\omax\osw_n$ with
$\theta^{-1}$ completely positive. Therefore, Lemma \ref{diagram lemma} asserts that
leftmost arrow on the diagram is a complete quotient map if and only if the bottom arrow is a complete
order isomorphism. The following result captures this fact and two additional equivalences.

\begin{theorem} \label{w-kp} The following statements are equivalent:
\begin{enumerate}
\item\label{w-kp-1} the map $\phi\otimes\phi:\mn\omin\mn\rightarrow\osw_n\omin\osw_n$ is a complete quotient map;
\item\label{w-kp-2} $\osw_n\omin\osw_n=\osw_n\omax\osw_n$;
\item\label{w-kp-3} $\ose_n\omin\ose_n=\ose_n\omax\ose_n$;
\item\label{w-kp-4} $\osw_n$ has property $\pstarn$.
\end{enumerate}
If any of these equivalent statements is true for every $n\in\mathbb N$, 
then the Kirchberg Problem has an affirmative solution.
\end{theorem}

\begin{proof} The equivalence of \eqref{w-kp-1} and \eqref{w-kp-2} is, as mentioned above,
a consequence of Corollary \ref{max quo cor} and Lemma \ref{diagram lemma}.

Statements \eqref{w-kp-2} and \eqref{w-kp-3} are equivalent by duality. Indeed, by
Proposition \ref{e-wd}, the 
operator systems $\osw_n^d$ and $\ose_n$ are completely order isomorphic.
Therefore, the hypothesis $\ose_n\omin\ose_n=\ose_n\omax\ose_n$ is equivalent to
$\osw_n^d\omin\osw_n^d=\osw^d_n\omax\osw^d_n$ and, by passing to duals
(Proposition \ref{fd dual tensor coi}), is equivalent to
$\osw_n\omax\osw_n=\osw_n\omin\osw_n$.

 Theorem \ref{pstar thm}\eqref{pstar-4} asserts that $\osw_n$ has property $\pstarn$
 if and only if $(\mn/\kj_n)\omin\osw_n=(\mn/\kj_n)\omax\osw_n$. Since $\mn/\kj_n$
 is completely order isomorphic to $\osw_n$, we deduce that statements \eqref{w-kp-2}
 and \eqref{w-kp-4} are equivalent.
 
Suppose now that any one of the equivalent conditions holds, for every $n\in\mathbb N$. 
By Theorem \ref{pstar thm}\eqref{pstar-4}, 
$\cstar(\fnl)\omin\osw_n=\cstar(\fnl)\omax\osw_n$. But this means that $\cstar(\fnl)$
has property $\pstarn$ and so again Theorem \ref{pstar thm}\eqref{pstar-4}
is invoked to conclude that 
$\cstar(\fnl)\omin\cstar(\fnl)=\cstar(\fnl)\omax\cstar(\fnl)$. As this is true for all $n$, we obtain
$\cstar(\fni)\omin\cstar(\fni)=\cstar(\fni)\omax\cstar(\fni)$ by the direct limit argument used in the
proof of Corollary \ref{k thm}.
\end{proof} 
 
In moving from the operator system $\osw_n$ to the operator system $\osw_{n-1}$, it
it is interesting to contrast Theorem \ref{w-kp} above with Theorem \ref{s is better} below.

\begin{theorem}\label{s is better} {\rm (\cite{kavruk--paulsen--todorov--tomforde2010})}
The Kirchberg Problem has an affirmative solution
if and only if
$\oss_{n}\omin\oss_{n}=\oss_{n}\oc\oss_{n}$ for every $n\in\mathbb N$.
\end{theorem}

\begin{proof} If $\oss_{n}\omin\oss_{n}=\oss_{n}\oc\oss_{n}$ for every $n\in\mathbb N$, then 
$\cstar(\fni)$ has the weak expectation property \cite[Theorem 9.14]{kavruk--paulsen--todorov--tomforde2010}
which in turn implies that $\cstar(\fni)\omin\cstar(\fni)=\cstar(\fni)\omax\cstar(\fni)$ \cite[Proposition 1.1(i)]{kirchberg1993}.

Conversely, if $\cstar(\fni)\omin\cstar(\fni)=\cstar(\fni)\omax\cstar(\fni)$, then $\cstar(\fni)$ has WEP
\cite[Proposition 1.1(iii)]{kirchberg1993}, and so $\oss_{n}\omin\oss_{n}=\oss_{n}\oc\oss_{n}$ for every $n\in\mathbb N$
\cite[Theorem 9.14]{kavruk--paulsen--todorov--tomforde2010}.
\end{proof}
 
 We now turn to our second multilinear algebra problem.

\begin{theorem} \label{s1-kp} The following statements are equivalent:
\begin{enumerate}
\item\label{s1-kp-1} $\phi\otimes\phi: \tn \omin \tn \to \oss_{n-1}\omin \oss_{n-1}$ is a complete quotient map;
\item\label{s1-kp-2} $\osu_n\omax\osu_n\subseteq_{\rm coi}\osv_n\omax\osv_n$.
\end{enumerate}
If either of these equivalent statements is true for every $n\in\mathbb N$, 
then the Kirchberg Problem has an affirmative solution.
\end{theorem}

\begin{proof} Assume \eqref{s1-kp-1} and let $\mu=\phi\otimes\phi$. Thus,
$\mu$ is a complete quotient map.
By Theorem \ref{dual of tn}, $\tn^d =\osu_n$ and $\oss_{n-1}^d = \osv_n$ are complete order isomorphisms.  
Thus, if $\mu$ is a complete quotient map, then
\[ 
\mu^d: (\oss_{n-1}\omin\oss_{n-1})^d \to (\tn \omin \tn)^d 
\]
is a complete order injection, but 
\[ (\oss_{n-1}\omin\oss_{n-1})^d = \osu_n\omax\osu_n \]
and
\[ 
(\tn \omin \tn)^d = \osv_n \omax \osv_n \,. 
\]

Conversely, assume \eqref{s1-kp-2}. That is, $\osu_n\omax\osu_n\subseteq_{\rm coi}\osv_n\omax\osv_n$, and so
\[
 \oss_{n-1}^d\omax\oss_{n-1}^d\,=\,\osu_n\omax\osu_n\,\subseteq_{\rm coi}\,\osv_n\omax\osv_n=\tn^d\omax\tn^d\,.
\]
Denote this complete order injection $ \oss_{n-1}^d\omax\oss_{n-1}^d\rightarrow \tn^d\omax\tn^d$ by $\vartheta$.
Hence, in passing to duals, we obtain a complete quotient map
\[
\vartheta^d:\tn\omin\tn=(\tn^d\omax\tn^d)^d\rightarrow(\oss_{n-1}^d\omax\oss_{n-1}^d)^d=
\oss_{n-1}\omin\oss_{n-1}.
\]
Since $\vartheta^d = \mu$,  statement \eqref{s1-kp-1} follows.

Assume that one of \eqref{s1-kp-1} or \eqref{s1-kp-2} is true for all $n\in\mathbb N$.
By hypothesis, each $y\in \mpee(\oss_{n-1}\omin\oss_{n-1})_+$ is given by $\mu^{(p)}(x)$ for some
$x\in\mpee(\tn\omin\tn)_+$. Recall from Proposition \ref{tn is (min,c)-nucl} that $\tn$
is (min,c)-nuclear; therefore, in particular, $\tn\omin\tn=\tn\oc\tn$.
Thus, since $\mpee(\tn\omin\tn)_+=\mpee(\tn\oc\tn)_+$ and 
because $\mu:\tn\oc\tn\rightarrow\oss_{n-1}\oc\oss_{n-1}$
is completely positive, $y\in \mpee(\oss_{n-1}\oc\oss_{n-1})_+$, which implies that 
$\oss_{n-1}\omin\oss_{n-1}=\oss_{n-1}\oc\oss_{n-1}$. 
As this is true for all $n\in\mathbb N$, Theorem \ref{s is better} completes the argument.
\end{proof}

\section{C$^*$-Algebras with the Weak Expectation Property}\label{S:WEP}

By considering the positive liftings in Proposition \ref{positive} in the case of the operator system $\tn$, we are
led to consider a special case of property $\pstarn$, which we call property
$\psn$.

\begin{definition} An operator system $\oss$ is said to have \emph{property $\psn$} 
for a fixed $n\in\mathbb N$ if, for every $p\in\mathbb N$ and all
$S_{i}\in \mpee(\oss)$, where $1-n\leq i\leq n-1$ and $j\neq i$, for which
\[
\sum_{i=1-n}^{n-1}u_i\otimes S_{i}\;\geq\;0 \mbox{ in }\; \cstar(\fnl)\omin\mpee(\oss)\,,
\]
then for every $\varepsilon>0$ there exist $R_{ij}^\varepsilon\in  \mpee(\oss)$, $1\leq i,j\leq n$,
such that
\begin{enumerate}
\item $R_\varepsilon=[R_{ij}^\varepsilon]_{1\leq i,j\leq n}$ is positive in $\mn(\mpee(\oss))$,
\item $R_{ij}^\varepsilon=0$ for all $|i-j|\geq2$, $R_{i,i+1}^\varepsilon=S_{i}$, and $R_{i+1,i}^\varepsilon=S_{-i}$ for all $i$, and
\item $\displaystyle\sum_{i=1}^n R_{ii}^\varepsilon \,=\,S_0+\varepsilon 1_{\mpee(\oss)}$.
\end{enumerate}
We say that $\oss$ has \emph{property $\ps$} if it has property $\psn$ for every $n \in \mathbb N.$
\end{definition}

\begin{proposition}\label{prop psn} The following statements are equivalent for an operator system $\oss$:
\begin{enumerate}
\item\label{psn-1} $\oss$ has property $\psn$;
\item\label{psn-2} $\phi\otimes{\rm id}_\oss:\tn\omin\oss\rightarrow\oss_{n-1}\omin\oss$
is a complete quotient map.
\end{enumerate}
\end{proposition}

\begin{proof} To set up the argument, note that $\mpee(\osr\omin\ost)$ and $\osr\omin\mpee(\ost)$ are completely order isomorphic
operator systems for any $p\in\mathbb N$ and operator systems $\osr$ and $\ost$. If
 $Z\in \oss_{n-1}\otimes \mpee(\oss)$ is arbitrary, then 
$Z=\displaystyle\sum_{i=1-n}^{n-1}u_i\otimes S_{i}$ for some $S_i\in\mpee(\oss)$. In this case
$Z=[\phi\otimes {\rm id}_\oss]^{(p)}(X)$, where (one choice of)
$X\in \tn\otimes\mpee(\oss)$ is given by
\[
X\,=\,\sum_{j=1}^n E_{jj}\otimes S_0 \,+\, \sum_{i=1}^n E_{i, i+1}\otimes nS_i\,+\,\sum_{i=1}^n E_{i+1,i}\otimes nS_{-i}\,.
\]
Via the First Isomorphism Theorem, 
$Z$ is the image of the quotient element $\dot{X}$.

Choose $p$ and suppose that  $Z\in \oss_{n-1}\omin \mpee(\oss)$
is positive. By definition, the element $\dot{X}$ is positive
if and only if
for every $\varepsilon>0$ there is a $K_\varepsilon=[K_{ij}^\varepsilon]_{i,j=1}^n\in\tn\otimes\mpee(\oss)$ such that
(i) $K_\varepsilon$ is diagonal,
(ii) $\sum_{i=1}^n K_{ii}^\varepsilon=0$, and (iii)
$\tilde R_\varepsilon=
\varepsilon1+X+K_\varepsilon$ is positive in $\tn\omin\mpee(\oss)$. But $\tilde R_\varepsilon$ is positive
if and only if $R_\varepsilon=\frac{1}{n}\tilde R_\varepsilon$ is positive. Note that $R_\varepsilon=[R_{ij}^\varepsilon]_{i,j}$ is tridiagonal
with $\sum_i R_{ii}^\varepsilon=\varepsilon1_{\mpee(\oss)}+S_0$, $R_{i,i+1}^\varepsilon=S_i$, and $R_{i+1,i}^\varepsilon=S_{-i}$.

Thus, if $\oss$ has property $\psn$, 
then the positivity of $R_\epsilon$ for every $\varepsilon>0$  implies that $\dot{X}$ is positive. Conversely, if $\phi\otimes{\rm id}_\oss$
is a complete quotient map, then the positivity of $Z$ implies that of $\dot{X}$ and, hence, the positivity of $R_\varepsilon$ also, for every $\varepsilon>0$.
\end{proof}

\begin{theorem}\label{WEP thm} The following statements are equivalent for a unital C$^*$-algebra
$\csta$:
\begin{enumerate}
\item\label{wep-1} $\csta$ has property $\ps$;
\item\label{wep-2} $\csta$ has WEP.
\end{enumerate}
\end{theorem}

\begin{proof} 
Assume \eqref{wep-1} holds and consider
the following commutative diagram of vector spaces and linear transformations:
\[
\begin{CD}
\tn\omin\csta   @>{ \cong}>>\tn\omax\csta \\
@V{\phi\otimes{\rm id}}VV           @VV{\phi\otimes{\rm id}}V  \\
\oss_{n-1}\omin\csta   @>>{{\rm id}_{\oss_{n-1}\otimes\csta}}> \oss_{n-1}\omax\csta\,.
\end{CD}
\]
The top arrow is a complete order isomorphism by Proposition \ref{tn is (min,c)-nucl} and 
\cite[Theorem 6.7] {kavruk--paulsen--todorov--tomforde2009}, and the leftmost arrow ($\phi\otimes{\rm id}$)
is a complete quotient map, by hypothesis. 
By Proposition~\ref{max quo}, the righthand arrow is also a quotient
map. Now apply Lemma~\ref{diagram lemma} to conclude that the bottom
arrow is a complete order isomorphism.
 
Thus, $\oss_{n-1}\omin\csta=\oss_{n-1}\omax\csta$, for all $n\in\mathbb N$. By the direct limit argument
of Corollary \ref{k thm} we arrive at $\cstar(\fni)\omin\csta=\cstar(\fni)\omax\csta$, which is equivalent to $\csta$
having WEP. Thus, statement \eqref{wep-1} is established.

Next assume \eqref{wep-2}. Thus, $\cstar(\fni)\omin\csta=\cstar(\fni)\omax\csta$, by \cite[Proposition 1.1(iii)]{kirchberg1993},
and so $\cstar(\fn)\omin\csta=\cstar(\fn)\omax\csta$, for every $n$ \cite[Lemma 7.5]{kavruk--paulsen--todorov--tomforde2010}.
Therefore, for every $n$,
\[
\oss_n\omin\csta\,\subseteq_{\rm coi}\cstar(\fn)\omin\csta\,=\,\cstar(\fn)\omax\csta\,.
\]
By the inclusion and equality above, 
\[
\oss_n\omin\csta\,=\,\oss_n\oc\csta\,=\,\oss_n\omax\csta\,.
\]
Every unital C$^*$-algebra is (c,max)-nuclear \cite[Theorem 6.7] {kavruk--paulsen--todorov--tomforde2009}. Thus,
using the fact that $\tn$ is (min,c)-nuclear, we deduce that
\[
\tn\omin\csta\,=\,\tn\oc\csta\,=\,\tn\omax\csta\,.
\]
We are therefore led to the following commutative diagram:
\[
\begin{CD}
\tn\omin\csta   @>{ \cong}>>\tn\omax\csta \\
@V{\phi\otimes{\rm id}}VV           @VV{\phi\otimes{\rm id}\quad\mbox{\rm (complete quotient map)}}V  \\
\oss_{n-1}\omin\csta   @>>{\cong}> \oss_{n-1}\omax\csta\,.
\end{CD}
\]
Hence, by Lemma \ref{diagram lemma} the map $\phi\otimes{\rm id}:\tn\omin\csta\rightarrow\oss_{n-1}\omin\csta$
is a complete quotient map for every $n$, and so $\csta$ has property $\ps$, which proves \eqref{wep-1}.
\end{proof}
 
By similar methods we can also prove:

\begin{theorem} If $\oss_{n-1}$ has property $\psn$ for every $n\in\mathbb N$, then Kirchberg's Problem has an affirmative solution.
\end{theorem}

\section{Injective Envelopes}\label{S:ie}
In the earlier sections we have seen that operator system quotients of matrix spaces can have large $C^*$-envelopes. 
In this section, we examine their injective envelopes. 

Recall that the injective envelope \cite[Chapter 15]{Paulsen-book}
of a unital C$^*$-algebra $\csta$ is denoted by $\ia$. An AW$^*$-algebra is a unital 
C$^*$-algebra $\cstb$ with the property that, for any nonempty subset $\mathcal X\subset\cstb$, there is a
projection $e\in\cstb$ such that ${\rm ann}_R(\mathcal X)=\{ey\,|\,y\in\cstb\}$, where
\[ {\rm ann}_R(\mathcal X)\;=\;\{b\in\cstb\,|\,xb=0,\;\forall\,x\in\mathcal X\}\,.\]
Like von Neumann algebras, AW$^*$-algebras admit a decomposition into direct sums of AW$^*$-algebras 
of types I, II, and III. An AW$^*$-factor is an AW$^*$-algebra with trivial centre. 
The theory of AW$^*$-algebras is relevant
to the study of injective envelopes because every injective C$^*$-algebra is monotone complete and
every monotone complete C$^*$-algebra is an AW$^*$-algebra. 
 
 \begin{theorem} If $\csta\neq\mathbb C$ is a unital, separable, prime C$^*$-algebra with only trivial projections, then its injective envelope
 $\ia$ is a type III AW$^*$-factor with no normal states.
 \end{theorem}
 
 \begin{proof} Under the hypothesis given, the injective envelope of $\csta$ cannot be a W$^*$-algebra
 (that is, a von Neumann algebra): for if $\ia$ were a von Neumann algebra, then the separability of $\csta$ would imply that 
 $\csta$ has a nontrivial projection (see \cite[Theorem 2.2(v)]{argerami--farenick2008} or \cite[\S3]{hamana1982c}), contrary
 to the hypothesis.
 
 As $\csta$ is separable, $\csta$ has a faithful representation as a unital C$^*$-subalgebra of $\bh$ for some separable Hilbert space $\hil$.
 The injective envelope of $\csta$ is obtained as follows \cite{Paulsen-book}: there is a ucp projection $\phi:\bh\rightarrow\bh$ 
mapping onto an operator system $\osi$ that contains $\csta$ as a subsystem and such that $\osi$ is an 
injective C$^*$-algebra under the Choi--Effros product
$x\circ y=\phi(xy)$, $x,y\in\osi$ \cite{choi--effros1977,Paulsen-book}. The injective envelope $\ia$ of $\csta$, 
when $\ia$ is viewed as a C$^*$-algebra, is precisely the operator system
$\osi$ with the Choi--Effros product $\circ$.
 
 Because $\hil$ is separable, $\bh$ has a faithful state $\omega$. Let $\varphi$ be the state on $\osi$ given by $\varphi(x)=\omega\left(\phi(x)\right)$.
 By the Schwarz inequality, $\varphi(x^*\circ x)=\omega\left(\phi(x^*x)\right)\geq\omega\left(\phi(x)^*\phi(x)\right)$, which implies that $\varphi$ is
 a faithful state on the AW$^*$-algebra $\ia$.
 
 The hypothesis that $\csta$ is prime (that is, no two nonzero ideals can have zero intersection) implies that the injective envelope $\ia$
 is an AW$^*$-factor \cite[Theorem 7.1]{hamana1981}. 
 We show that the only type of factor that $\ia$ could possibly be is a factor of type III.
 
 If $\ia$ were a factor of type I, then $\ia$ would be coincide with $\bh$ for some Hilbert space $\hil$, implying that $\ia$
 is a von Neumann algebra, which we have argued is not possible. If $\ia$ were a factor of type II (finite) or II${}_\infty$, then
 the fact that $\ia$ admits a faithful state implies that $\ia$ is a von Neumann algebra \cite{elliott--saito--wright1983},
 which again is not possible.
 Therefore, $\ia$ must be of type III.
 
 Finally, if $\ia$ were to have a normal state, then $\ia$ would have a direct summand which is a von Neumann algebra. But the
 fact that $\ia$ is a factor discounts the existence of a nonzero direct summand, and because $\ia$ is not a W$^*$-algebra, we conclude
 that $\ia$ has no normal states.
 \end{proof}
 
 \begin{corollary} The injective envelope of $\cstar(\fn)$, where $n\geq2$, is a type III AW$^*$-factor with no normal states.
 \end{corollary}
 
 \begin{proof} If $n\geq 2$, then $\cstar(\fn)$ is unital, separable, primitive (and therefore prime), and has only trivial projections \cite{choi1980}.
 \end{proof}

 \begin{corollary} The injective envelopes of $\mn/\kj_n$ and $\tn/\kj_n$ coincide. When $n\geq2$, their injective envelope is a type III AW$^*$-factor with no normal states.
 \end{corollary}
 
 \begin{proof} The injective envelope of an operator system $\oss$ and its C$^*$-envelope $\cstare(\oss)$ coincide. Because
 $\cstare(\tn/\kj_n) = \cstare(\mn/\kj_n)=\cstar(\fnl)$ (Theorem \ref{mn mod jn}), the result follows.
 \end{proof}

\section*{Acknowledgement}

This work was undertaken at Institut Mittag-Leffler (Djursholm, Sweden) in autumn 2010. The authors wish to acknowledge the support and hospitality
of the institute during their stay there. The authors also wish to thank Ali S.~Kavruk for pointing out a mistake in an earlier version of this paper, 
Ivan Todorov for useful commentary, and the referee for a very careful review of the manuscript and for drawing our attention to
reference \cite{boca1997}.
The work of the first author is supported in part by NSERC.


\end{document}